\documentclass{amsart}
\usepackage{amsmath}
\usepackage{amssymb}
\usepackage{amsthm}
\usepackage{commath}
\usepackage{natbib}
\usepackage[margin=4cm]{geometry}
\usepackage{microtype}
\usepackage{multirow}
\usepackage{rotating}
\usepackage{booktabs}
\usepackage{lscape}

\newtheorem{theorem}{Theorem}
\newtheorem{lemma}{Lemma}
\newtheorem{corollary}{Corollary}

\newcommand{\sgn}{\mathrm{sign}}

\begin{document}


\title{Oracle Inequalities for High-dimensional Panel Data Models}

 \author{Anders Bredahl Kock\\ Aarhus University and CREATES}
 \thanks{I am grateful to Mehmet Caner and Peter Phillips for urging me to pursue the ideas of this paper. I would also like to thank seminar participants at PUC-Rio and participants at the European Meeting of Statisticians 2013 in Budapest for helpful comments and suggestions. Financial support from the Danish National Research Foundation (DNRF78) is gratefully acknowledged.\\
e-mail: \texttt{akock@creates.au.dk}. Address: Aarhus University and CREATES, Fuglesangs Alle 4, 8210 Aarhus V}

\maketitle

\begin{abstract}
This paper is concerned with high-dimensional panel data models where the number of regressors can be much larger than the sample size. Under the assumption that the true parameter vector is sparse we propose a panel-Lasso estimator and establish finite sample upper bounds on its estimation error under two different sets of conditions on the covariates as well as the error terms. In particular, we allow for heteroscedastic and non-gaussian error terms which are weakly dependent over time. Upper bounds on the estimation error of the unobserved heterogeneity are also provided under the assumption of sparsity. Next, we show that our upper bounds are essentially optimal in the sense that they can only be improved by multiplicative constants. These results are then used to show that the Lasso can be consistent in even very large models where the number of regressors increases at an exponential rate in the sample size. Conditions under which the Lasso does not discard any relevant variables asymptotically are also provided.

In the second part of the paper we give lower bounds on the probability with which the adaptive Lasso selects the correct sparsity pattern in finite samples. These results are then used to give conditions under which the adaptive Lasso can detect the correct sparsity pattern asymptotically. We illustrate our finite sample results by simulations and apply the methods to search for covariates explaining growth in the G8 countries.

\vspace{.4cm}

\end{abstract}

\section{Introduction}


When building a model one of the first decisions one has to make is which, of potentially many,  variables are to be included in the model and which are to be left out. Often this decision is made based on a particular theory but different theories might suggest different explanatory variables and this leaves the researcher with a large set of potential variables. In fact, one may often have access to many more variables than observations rendering standard techniques inapplicable. Since this kind of high-dimensional data is becoming increasingly available, the last 10-15 years have witnessed a great deal of research into procedures that can handle such data sets. In particular, a lot of attention has been given to penalized estimators. The Lasso of \cite{tibshirani96} is the most prominent of these procedures and a lot of subsequent research has focussed on investigating the theoretical properties of the Lasso, see \cite{zhaoy06}, \cite{meinshausenb06}, \cite{bickelrt09}, \cite{bellonic11} and \cite{buhlmannvdg11} to mention just a few. The Lasso and related procedures have become popular since they are computationally feasible and perform variable selection and parameter estimation simultaneously. For recent reviews we refer to \cite{buhlmannvdg11}, \cite{bellonic11} and \cite{fan11}.

Lasso-type estimators have been used by \cite{BelloniCH13} to show that uniform inference on treatment effects is possible even after selecting among a large set of control variables. This is valid even if the controls have been imperfectly selected and in the presence of heteroscedastic and non-gaussian error terms. \cite{caner2009} used the Lasso in the context of GMM and has extended his procedure in \cite{canerz2014} to also allow for an increasing number of parameters.

Most focus in the literature has been on cross sectional data. However, objects (such as individuals, firms or countries) are often sampled repeatedly over time resulting in a panel data set. Panel data sets are attractive since they have several advantages compared to cross sectional or time series data. Firstly, they may decrease the collinearity between the covariates compared to pure time series models. Secondly, and perhaps more importantly, \cite{hsiao03} gives examples of statistical questions which can not be answered in the standard cross sectional model -- variation over time is needed. Finally, in empirical research, one often hears that a certain effect is found (or not found) due omitted variables that are correlated with the error term. Panel data allows us to control for omitted variables that can be hard to measure or get data on (such as the intelligence or quality of a worker). However, since panel data sets may often contain many variables it is important to have procedures that can deal with them in a theoretically sound and computationally feasible manner. We contribute by being the first to propose a panel-Lasso estimator and by establishing oracle inequalities for it and the adaptive Lasso in the high-dimensional linear fixed effects panel data model
\begin{align}
y_{i,t}=x_{i,t}'\beta^* +c_i^*+\epsilon_{i,t},\ i=1,...,N,\ t=1,...,T\label{model1} 
\end{align}
where $x_{i,t}$ is a $p_{N,T}\times 1$ vector of covariates and where $p_{N,T}$ is indexed by $N$ and $T$ to indicate that the number of covariates can increase in the sample size. In the sequel we shall omit this indexation. The $c_i^*$s are the unobserved time homogeneous fixed effects (such as intelligence, ability, motivation or perseverance of a person) while the $\epsilon_{i,t}$ are the error terms about which we shall be more specific later. In contrast to the standard linear regression model the fixed effect panel data model has two potential sources of high-dimensionality. First, even though theory may guide the researcher towards a set of potential explanatory variables to be included in $x_{i,t}$, large data sets are becoming increasingly available nowadays and one may not want to take a strong stand a priori on which variables to include in the model and which to leave out. This implies that $x_{i,t}$ can be a very long vector -- potentially much longer than the sample size. On the other hand, only a few variables in $x_{i,t}$ might be relevant for explaining $y_{i,t}$ meaning that the vector $\beta^*$ is sparse. 

The second source of high-dimensionality is the vector $c^*=(c_1^*,...,c_N^*)$. Note that this always has as many entries as there are cross sectional observations in the data set ($N$). Often the unobserved heterogeneity $c_{i}$ is simply removed by a differencing or demeaning procedure. In the context of a high-dimensional panel data model this was the route chosen by \cite{kock13}. However, we show that the unobserved heterogeneities can be estimated and provide finite-sample upper bounds on the estimation error. Just like $\beta^*$, $c^*$ might be a sparse vector. A first example of this could be the intelligence or ability of a worker having an effect on income only for certain individuals when modeling the deviation of their income from the national mean income. More precisely, one might conjecture that for a large group of workers of "average" ability their ability does not induce any fluctuations away from the mean income. This big group of workers would have $c_i^*=0$. On the other hand, persons with above average ability may also have incomes that are above the average corresponding to $c_i^*>0$. A similar argument would imply that workers with abilities far below the average would have $c_i<0$.

A second example could be the effect of culture when modeling the growth of developing countries. It is our goal to investigate the properties of the panel-Lasso for fixed effects panel data models. We shall see that the panel-Lasso can estimate the two parameter vectors almost as precisely as if the true sparsity pattern had been known and only the relevant variables had been included from the outset. For the adaptive Lasso we show that it selects the correct sparsity pattern with high probability. In particular, we 

\begin{enumerate}
	\item are the first to provide \textit{nonasymptotic} oracle inequalities for the estimation error of the proposed panel-Lasso for $\beta^*$ and $c^*$ under different sets of moment/tail assumptions on the covariates and the error terms which allow for heteroscedasticity and non-gaussianity. More precisely, for a given sample size we provide upper bounds on the estimation error which hold with at least a certain probability. These upper bounds are of the same order as if an oracle had revealed the true model prior to estimation and one had only included the relevant variables from the outset. Finally, the bounds are uniform over certain subsets of the parameter space.
	\item In the first of our settings we allow for much heavier tails than the usual sub-gaussian ones. We also stress that due to the presence of two parameter vectors we develop a new trick to obtain the optimal oracle inequalities. Applying conventional techniques would lead to sub-optimal bounds. 
	\item show that our bounds are optimal in the sense that they can at most be improved by a multiplicative constant. 
	\item use the nonasymptotic bounds to give a set of sufficient conditions under which the Lasso estimates $\beta^*$ and $c^*$ consistently. It turns out that the Lasso can be consistent in even very high-dimensional models. We also provide conditions under which the Lasso does not discard any relevant variables, i.e. conditions under which it can be used as a strong initial screening device removing irrelevant variables and thus reducing the dimension of the model.  
	\item establish nonasymptotic lower bounds on the probability with which the adaptive Lasso unveils the correct sparsity pattern.
	\item use the nonasymptotic bounds to give conditions under which the adaptive Lasso detects the correct sparsity pattern asymptotically.
	\item propose an efficient algorithm to implement the Lasso and the adaptive Lasso in panel data models which reduces the estimation problem to a standard Lasso one.
	\item introduce a new restricted eigenvalue condition similar in spirit to \cite{bickelrt09} and show how this can be valid even for data with non-gaussian, non-independent rows, hence extending the work of \cite{raskuttiwy10} and \cite{vershynin11} and the state of the art in \cite{rudelsonz2013}. We stress again that the proof of our Theorem \ref{thm1} is also different than the one for the Lasso in the plain cross sectional model due to the presence of two parameter vectors which have to be treated separately in order to obtain tight finite sample upper bounds on the estimation error of both parameter vectors.   
	\item illustrate the methods by means of simulations and a real data example. In particular, the validity of the neo-clasical growth hypothesis is investigated by means of a large panel data set.
\end{enumerate}
We believe that these results will be very useful for applied researchers since they provide tools with which very large panel data sets can be handled in a theoretically sound an computationally feasible manner without reducing the dimension of the model in an ad hoc way prior to estimation. To be precise, we show that one can simply include all relevant variables and still estimate the coefficients as precisely as if only the relevant variables had been included from the outset.

The rest of the paper is organized as follows: Section \ref{setup} introduces relevant notation and the panel Lasso. Section \ref{Lasso} provides a range of non-asymptotic oracle inequalities for the Lasso while Section \ref{asymLasso} uses these inequalities to give asymptotic results for it. Next, Section \ref{ALasso} is concerned with finite sample probabilities of the adaptive Lasso selecting the correct sparsity pattern. It also gives sufficient conditions for when this probability tends to one asymptotically. Section \ref{MC} provides a simulation study while Section \ref{Emp} contains an application to growth in the G8 countries. Finally, Section \ref{concl} concludes while all proofs are deferred to the appendix. 
  
\section{Setup and notation}\label{setup}
Let $J_1=\cbr[0]{j: \beta^*_j\neq 0}\subseteq\cbr[0]{1,...,p}$ and $J_2=\cbr[0]{i: c^*_i\neq 0}\subseteq\cbr[0]{1,...,N}$ be the sets of active covariates and unobserved heterogeneities, respectively. $\beta_{\min}=\min\cbr[0]{|\beta_j^*|:j\in J_1}$ and $c_{\min}=\min\cbr[0]{|c_j^*|:j\in J_2}$ are the smallest nonzero entries of $\beta^*$  and $c^*$, respectively. Denote by $\gamma^*=({\beta^*}',{c^*}')'$ and $J=J_1\cup J_2\subseteq\cbr[0]{1,...,N+p}$\footnote{Here $J_1\cup J_2$ is understood as $J_1\cup (J_2+p)$ where $J_2+p=\cbr{s=r+p: r\in J_2}$ such that $J_1\cup J_2\subseteq\cbr[0]{1,...,p+N}$. $J$ shall be used to index $p+N\times 1$ vectors.}. For any set $A$, $|A|$ denotes its cardinality while $A^c$ denotes its complement. In particular, $|J_1|=s_1,\ |J_2|=s_2$ and $|J|=s$.

For any $x\in \mathbb{R}^n$, $\enVert[0]{x}=\sqrt{\sum_{i=1}^nx_i^2}$, $\enVert[0]{x}_{\ell_1}=\sum_{i=1}^n |x_i|$ and $\enVert[0]{x}_{\ell_\infty} =\max_{1\leq i\leq n} |x_i|$ denote $\ell_2$-, $ \ell_1$- and $\ell_{\infty}$-norms, respectively. A couple of times we shall also make use of the $\ell_0$-norm $\enVert{x}_{\ell_0}=\sum_{i=1}^n\mathbf{1}_{\cbr[0]{x_i\neq 0}}$ which is simply the number of non-zero entries of $x$ For a random variable $U$, $\enVert[0]{U}_{L_r}=(E|U|^r)^{1/r}$ denotes its $L_r$-norm and for a symmetric square matrix $M$, $\phi_{\min}(M)$ and $\phi_{\max}(M)$ denote the minimal and maximal eigenvalues of $M$.

For any vector $x\in\mathbb{R}^n$ and subset $A$ of $\cbr[0]{1,...,n}$, $x_J$ denotes the vector in $\mathbb{R}^{|J|}$ only consisting of the elements indexed by $A$.  For a matrix $R$, $R_A$ denotes the submatrix only containing the columns indexed by $A$ while $R_{A,B}$ denotes the submatrix with rows indexed by $A$ and columns indexed by $B$.
Next, for any two real numbers $a$ and $b$, $a\wedge b=\min(a,b)$ and $a\vee b=\max(a,b)$. For any $x\in\mathbb{R}^n$, $\sgn(x)$ denotes the sign function applied to each component of $x$.

Since our primary focus is high-dimensional models we shall sometimes tacitly assume that $p,N\geq e$ for the sole reason of keeping the presentation simple\footnote{Here $e$ denotes Euler's constant.}.

Define $X_i=(x_{i,1},...,x_{i,T})'$ and $X=(X_1',...,X_N')'$. Letting $\iota$ denote the $T\times 1$ vector of ones, set $D=I_N\otimes \iota$ (where $\otimes$ denotes the Kronecker product) and define the $NT\times (p+k)$ matrix $Z=(X,D)$. We shall refer to the $j$th column of $X$ by $x_j,\ j=1,...,p$ and to the $i$th column of $D$ by $d_i,\ i=1,...,N$. Defining $y_i=(y_{i,1},...,y_{i,T})'$ and $\epsilon_i=({\epsilon_{i,1},...,\epsilon_{i,T}})'$ for $i=1,...,N$ and setting $y=(y_1',...,y_N')'$ as well as $\epsilon=(\epsilon_1',...,\epsilon_N')'$  one may equivalently write (\ref{model1}) as
\begin{align*}
y=Z\gamma^*+\epsilon.
\end{align*}
The properly scaled Gram matrix of $Z$ will turn out to play an important role in the sequel.

\subsection{The panel Lasso}
The panel Lasso estimates $\gamma^*=({\beta^*}',{c^*}')'$ by minimizing the following objective function
\begin{align}
L(\beta,c)
&=
\sum_{i=1}^N\sum_{t=1}^T\del[1]{y_{i,t}-x_{i,t}'\beta-c_i}^2+2\lambda_{N,T}\sum_{k=1}^p\envert[0]{\beta_k}+2\mu_{N,T}\sum_{i=1}^N\envert[0]{c_i}\\
&=
\enVert{y-Z\gamma}^2+2\lambda_{N,T}\enVert{\beta}_{\ell_1}+2\mu_{N,T}\enVert{c}_{\ell_1}.\label{Lassoobj}
\end{align}
The Lasso estimator, denoted $\hat{\gamma}=(\hat{\beta}', \hat{c}')'$, is the solution of a minimization problem which is the sum of the usual least squares objective function plus two terms that penalize $\beta_k$ and $c_i$ for being different from 0. The size of the penalty is determined by the terms $\lambda_{N,T}$ and $\mu_{N,T}$. The larger these are, the more will the entries of $\hat{\beta}$ and $\hat{c}$ be shrunk towards zero. As will be seen later, two different regularization sequences ($\lambda_{N,T}$ and $\mu_{N,T}$) are needed to establish desirable properties of $\hat{\gamma}=(\hat{\beta}', \hat{c}')'$. This is due to the fact that common covariates and the fixed effects are of different orders of magnitude. This is in contrast to standard cross sectional models where all covariates are usually assumed to be of the same order of magnitude. Hence, we contribute by providing a technique which allows one to establish separate oracle inequalities for subvectors of different orders of magnitude.  On an intuitive level this technique is necessary due to the fact that the number of effective observations for each $\beta_k,\ k=1,...,p$ is $NT$ while it only is $T$ for each $c_i\ i=1,...,N$.  

Also note that if one knows a priori that certain variables are relevant one can choose not to penalize these which results in these variables being retained in the model. As an extreme case of this one might work on a problem where one does not believe that the fixed effects are sparse. In that case, if one is not interested in estimating the fixed effects, we refer to the methodology of \cite{kock13} who takes a classical approach and simply differences out the fixed effects without estimating them. We shall proceed by assuming that $\beta^*$ and $c^*$ are sparse but keep in mind that one can choose not to penalize certain parameters if these are known to be non-zero. 

\subsection{The panel restricted eigenvalue condition}
Since we are primarily interested high-dimensional models the properly scaled Gram matrix of $Z$ will often be ill-behaved or even singular. However, \cite{bickelrt09} observed for the standard linear regression model that the Lasso does not need the smallest eigenvalue of the scaled Grammian of $Z$ to be strictly positive in order to derive useful upper bounds on the estimation error.  In particular, it suffices that a so-called \textit{restricted eigenvalue} is bounded away from 0. We shall see next that a similar, though slightly more involved, observation can be made for the panel Lasso.

Let $S=\bigl(\begin{smallmatrix}\sqrt{NT}\mathbf{I}_{p}&0\\ 0&\sqrt{T}\mathbf{I}_{N}\end{smallmatrix} \bigr)$ and set $\psi_{N,T}=S^{-1}Z'ZS^{-1}$. If $p+N>NT$ it is well known that 
\begin{align*}
\min_{\delta\in\mathbb{R}^{p+N}\setminus \cbr[0]{0}}\frac{\delta'\Psi_{N,T}\delta}{\enVert[0]{\delta}^2}
=
\min_{\delta\in\mathbb{R}^{p+N}\setminus \cbr[0]{0}}\frac{\enVert[0]{ZS^{-1}\delta}^2}{\enVert[0]{\delta}^2}
=0.
\end{align*}
In this case ordinary least squares is infeasible. However, for the Lasso it turns out that we do not need to minimize the above Rayleigh-Ritz ratio over all of $\mathbb{R}^{p+N}$ -- it suffices to minimize over a subset implying that the minimum can be non-zero even when $\Psi_{N,T}$ is not of full rank. More precisely, letting $\delta^1$ be $p\times 1$ and $\delta^2$ be $N\times 1$ with $\delta=({\delta^1}',{\delta^2}')'$ and $R_1\subseteq \cbr[0]{1,...,p}$ as well as $R_2\subseteq \cbr[0]{1,...,N}$ we define the $RE(r_1,r_2)$ panel restricted eigenvalue as
\begin{align}
\kappa^2_{\psi_{N,T}}(r_1,r_2)
=
\min\Bigg\{\frac{\enVert{ZS^{-1}\delta}^2}{\enVert[0]{\delta}^2}: \delta \in \mathbb{R}^{p+N}\setminus\cbr[0]{0},\ |R_1|\leq r_1,\ |R_2|\leq r_2,\notag\\  
\frac{\lambda_{N,T}}{\sqrt{NT}}\enVert[1]{\delta_{R_1^c}^1}_{\ell_1}+\frac{\mu_{N,T}}{\sqrt{T}}\enVert[1]{\delta_{R_2^c}^2}_{\ell_1}\leq 3\frac{\lambda_{N,T}}{\sqrt{NT}}\enVert[1]{\delta_{R_1}^1}_{\ell_1}+3\frac{\mu_{N,T}}{\sqrt{T}}\enVert[1]{\delta_{R_2}^2}_{\ell_1}\Bigg\}>0\label{RE}.
\end{align}
The panel restricted eigenvalue condition looks similar to the one introduced in \cite{bickelrt09}. It extends it in that it allows for different penalty sequences for the two groups of parameters. Similarly, for 
\begin{align*}
\Gamma=\del[3]{\begin{matrix} E\del{\frac{X'X}{NT}} &0\\ 0& \mathbf{I}_{N}\end{matrix} }
\end{align*}
define 
\begin{align*}
\kappa^2(r_1,r_2)
=
\min\Bigg\{\frac{\delta'\Gamma\delta}{\enVert[0]{\delta}^2}: \delta \in \mathbb{R}^{p+N}\setminus\cbr[0]{0},\ |R_1|\leq r_1,\ |R_2|\leq r_2,\\ \frac{\lambda_{N,T}}{\sqrt{NT}}\enVert[1]{\delta_{R_1^c}^1}_{\ell_1}+\frac{\mu_{N,T}}{\sqrt{T}}\enVert[1]{\delta_{R_2^c}^2}_{\ell_1}\leq 3\frac{\lambda_{N,T}}{\sqrt{NT}}\enVert[1]{\delta_{R_1}^1}_{\ell_1}+3\frac{\mu_{N,T}}{\sqrt{T}}\enVert[1]{\delta_{R_2}^2}_{\ell_1}\Bigg\}.
\end{align*}
Note that for $\kappa^2>0$ it suffices that $\Gamma$ is of full rank which is a rather standard assumption and independent of whether $p+N<NT$ or not. I turns out that in order to get tight upper bounds on the estimation error of the Lasso $\kappa^2_{\Psi_{N,T}}$ should be as large as possible. In Lemma \ref{vdGB} in the appendix we show that  $\kappa^2_{\Psi_{N,T}}$ is close to $\kappa^2$ if $\Psi_{N,T}$ is close to $\Gamma$. Hence, it suffices that $\kappa^2$ is bounded away from zero and that $\Psi_{N,T}$ is close to $\Gamma$ in order to bound $\kappa^2_{\Psi_{N,T}}$ away from 0 with high probability. In Lemmas \ref{BboundLp} and \ref{Bboundexp} in the Appendix lower bounds on the probability with which $\kappa^2_{\Psi_{N,T}}>\kappa^2/2$ are provided using this idea for heavy- and light-tailedness assumptions on the covariates and the error terms. While the results for light-tailed (sub-gaussian) variables in Lemma \ref{Bboundexp} are to be expected in the light of previous results in the literature (see e.g. \cite{vershynin11}) the results on more heavy-tailed random variables in Lemma \ref{BboundLp} are to our knowledge new. What is new in both cases is that we are not dealing with exclusively independent variables.

\section{Results for the Lasso}\label{Lasso}
 
Before stating our first result we introduce the following two sets 
\begin{align*}
\mathcal{A}_{N,T}=\cbr[2]{\enVert[0]{X'\epsilon}_{\ell_\infty}\leq \frac{\lambda_{N,T}}{2},\ \enVert[0]{D'\epsilon}_{\ell_\infty}\leq \frac{\mu_{N,T}}{2}} \text{ and } \mathcal{B}_{N,T}=\cbr[1]{\kappa_{\Psi_{N,T}}^2\geq \kappa^2/2}. 
\end{align*}
The set $\mathcal{A}_{N,T}$ is the set where none of the covariates $X$ or $D$ are too highly correlated with the error term. This requirement limits the number of variables in $X$ and $D$. Working on the set $\mathcal{B}_{N,T}$ means restricting attention to settings where the restricted eigenvalue of $\Psi_{N,T}$ is not too small. 

Theorem \ref{thm1} below gives upper bounds on the estimation error of the Lasso on $\mathcal{A}_{N,T}\cap \mathcal{B}_{N,T}$ and will be our main tool to derive further bounds under more specific assumptions on the covariates and the error terms. It is worth emphasizing that it is a purely algebraic result without any probabilities attached to it yet.
\begin{theorem}\label{thm1}
On $\mathcal{A}_{N,T}\cap \mathcal{B}_{N,T}$ with $\kappa^2>0$ one has for any positive sequences $\lambda_{N,T}$ and $\mu_{N,T}$
\begin{align}
\enVert[1]{\hat{\beta}-\beta^*}
&\leq
\frac{8\lambda_{N,T}\sqrt{s_1}}{\kappa^2NT}+\frac{4\mu_{N,T}\sqrt{s_2}}{\kappa^2\sqrt{N}T}\label{IQthm1}
\end{align}
and
\begin{align}
\enVert[1]{\hat{c}-c^*}
&\leq
\frac{8\mu_{N,T}\sqrt{s_2}}{\kappa^2T}+\frac{4\lambda_{N,T}\sqrt{s_1}}{\kappa^2\sqrt{N}T}\label{IQ2thm1}.
\end{align}
\end{theorem}
We stress that the claims in Theorem \ref{thm1} are deterministic. Probabilities will be attached to the bounds once we have made statistical assumptions on the covariates and the error terms. 

The bounds in Theorem \ref{thm1} reveal that the further $\kappa^2$ is away from zero the more precisely can one estimate the parameters of the model. This is reasonable since it means that the problem is, in some sense, far from a singular one. However, the set $\mathcal{B}_{N,T}$ is clearly decreasing in $\kappa^2$, revealing a tradeoff between the sharpness of the upper bounds on the estimation error and the size of the set on which the bounds hold. The same tradeoff is present for $\lambda_{N,T}$ and $\mu_{N,T}$ -- the set $\mathcal{A}_{N,T}$ is increasing in both of these but the same is true for the upper bounds on the estimation error. Put differently, small values of $\lambda_{N,T}$ and $\mu_{N,T}$ give tight bounds on the estimation error but the bounds are only valid on a smaller set. We would also like to stress that the upper bounds in (\ref{IQthm1}) and (\ref{IQ2thm1}) are uniform over $\cbr[0]{\beta^*\in\mathbb{R}^p:\enVert{\beta^*}_{\ell_0}\leq s_1}\times \cbr[0]{c^*\in\mathbb{R}^n:\enVert{c^*}_{\ell_0}\leq s_2}$, i.e. the upper bounds are valid uniformly over certain $\ell_0$-balls. In particular, the only characteristic of the true parameter vector which matters is its number of non-zero entries. Their magnitude or position in the vector do not matter. 

The upper bounds in (\ref{IQthm1}) and (\ref{IQ2thm1}) are not obtained in an entirely standard manner but instead rely on a new trick which characterizes the upper bounds as the maximal solutions to a certain inequality. Application of the standard technique to establish oracle inequalities from the linear regression model would result in suboptimal bounds.

Our next two theorems investigate the tradeoff further under different sets of assumptions on the tail behaviour of the covariates and the error terms. First, we shall put forward the statistical assumptions of the panel data model:

\begin{itemize}
\item[A1] a) $\cbr[0]{X_{i},\epsilon_{i}}_{i=1}^N$ are identically and independently distributed \\ 
b) $X_i$ and $\epsilon_i$ are independent for $i=1,...,N$\\
c) $\cbr[1]{\epsilon_{i,t}, \mathcal{F}_{i,t}}_{t=1}^T$  with $\mathcal{F}_{i,t}=\sigma(\cbr[0]{\epsilon_{i,1},...,\epsilon_{i,t}})$ form a martingale difference sequence for all $i=1,...,N$ \footnote{Since A1a) and A1b) are purely distributional assumptions it is actually enough to assume that they are valid for $i=1$ since we have assumed identical distributedness across $i=1,...,N$}. 
\end{itemize}

Assumption A1a) is standard in the panel data literature, see e.g. \cite{wooldridge02} or \cite{arellano03}. We would like to stress that the requirement of the data being identically distributed is not necessary but it makes the exposition slightly easier. Part b) is also relatively standard but slightly stronger than $E(\epsilon_{it}|X_i)=0$ which is often assumed. However, for most applied work involving panel data it is hard to come up with realistic examples where $E(\epsilon_{it}|X_i)=0$  but $X_i$ and $\epsilon_i$ are not independent\footnote{Of course it is possible to construct examples where $E(\epsilon_{it}|X_i)=0$ but $X_i$ and $\epsilon_i$ are not independent. See e.g. \cite{stoianov87}.}. A1c) allows the error terms to be dependent over time for each individual but it is of course also valid in the case where they are independent. Note that we are \textit{not} assuming that $\cbr[0]{\epsilon_{1,t}}_{t=1}^T$ are identically distributed. In particular, they may be heteroscedastic. 

Furthermore, the upper bounds on the estimation errors in (\ref{IQthm1}) and (\ref{IQ2thm1}) as well as the probability with which they hold, depend on the number of moments the error terms and covariates possess. We shall give results under two different sets of conditions.

\begin{itemize}
	\item[A2a)] $E(|x_{1,t,k}|^r),\ E(|\epsilon_{1,t}|^r)<\infty$ for some $r\geq 2$ and $t=1,...,T,\ k=1,...,p$. Actually, we shall assume $\max_{1\leq t\leq T}E|x_{1,t,k}|^r\leq 1$ for all $k=1,...,p$. 
\end{itemize}
Assumption A2a) is a moment assumption stating that the covariates as well as the error terms possess $r$ moments.  $\max_{1\leq t\leq T}E|x_{1,t,k}|\leq 1$ for all $k=1,...,p$ is merely a normalization for technical convenience and to keep expressions simple. All results remain valid without this normalization.

\begin{itemize}
	\item[A2b)] $x_{1,t,k}$ and $\epsilon_{1,t}$ are uniformly subgaussian, i.e. there exist constants $C$ and $K$ such that $P\del[0]{|x_{1,t,k}|\geq t},\ P\del[0]{|\epsilon_{1,t}|\geq t}\leq \frac{1}{2}Ke^{-Ct^2}$ for all $1\leq t\leq T$ and $1\leq k\leq p$. 
\end{itemize}
Assumption A2b) controls the tail behaviour of the covariates and the error terms (and hence also its moments). It is a standard assumption in the high-dimensional econometrics literature and much more restrictive than A2a) which only assumes the existence of $r$ moments. However, we will see that the dimension of the models considered can be a lot larger under A2b) than under A2a).

We are now ready to transform the deterministic statement in Theorem \ref{thm1} into probabilistic ones. We stress that the bounds below are \textit{finite sample} bounds, i.e. for a given sample size we provide upper bounds on the estimation error that hold with at least a certain probability. First, we work under assumption A2a):

\begin{theorem}\label{thm2}
Let assumption A1) and A2a) be satisfied and assume that $\kappa^2>0$. Then, choosing $\lambda_{N,T}=4a_{N,T}p^{1/r}(NT)^{1/2}\max_{1\leq t\leq T}\enVert[0]{\epsilon_{1,t}}_{L_r}$ and $\mu_{N,T}=4a_{N,T}N^{1/r}T^{1/2}\max_{1\leq t\leq T}\enVert[0]{\epsilon_{1,t}}_{L_r}$ for any positive sequence $a_{N,T}$ one has $P\del[0]{\mathcal{A}_{N,T}\cap \mathcal{B}_{N,T}}\geq 1-2\del{\frac{C_r}{a_{N,T}}}^r-D_r\frac{(p^2+Np)(s_1+s_2)^{r/2}\del[0]{\frac{p}{N}\vee \frac{N}{p}}}{\kappa^rN^{r/4}}$ for constants $C_r$ and $D_r$ only depending on $r$. Furthermore, with at least this probability (i.e. on $\mathcal{A}_{N,T}\cap \mathcal{B}_{N,T}$),
\begin{align}
\enVert[1]{\hat{\beta}-\beta^*}
\leq
\frac{\xi_{N,T}}{\sqrt{NT}}\label{IQ1thm2}
\end{align}
and
\begin{align}
\enVert[1]{\hat{c}-c^*}
\leq 
\frac{\xi_{N,T}}{\sqrt{T}}\label{IQ2thm2}
\end{align}
where $\xi_{N,T}=32a_{N,T}\max_{1\leq t\leq T}\enVert[0]{\epsilon_{1,t}}_{L_r}\del[0]{p^{1/r}\sqrt{s_1}+N^{1/r}\sqrt{s_2}}/\kappa^2$.
\end{theorem}
First, note that the more moments the covariates and the error terms possess ($r$ large) the smaller can $\lambda_{N,T}$ and $\mu_{N,T}$ be chosen and hence the upper bounds on the estimation error are smaller in accordance with Theorem \ref{thm1}. Recall that $\lambda_{N,T}$ and $\mu_{N,T}$ should be chosen just large enough to ensure that $\mathcal{A}_{N,T}$ has a high probability. Since the summands in a generic entry of the vector $X'\epsilon$ are not independently distributed it requires different tools from the usual ones in the literature to find small values of $\lambda_{N,T}$ which ensure that $\mathcal{A}_{N,T}$ has a high probability. Had the summands been independent we could most likely have benefitted from the innovative use of self-normalization as promoted in \cite{belloni2012s}. However, even in the presence of independence across $t=1,...,T$, providing useful lower bounds on the probability of $\mathcal{B}_{N,T}$ is non-trivial since the state of the art in \cite{rudelsonz2013} also assumes the covariates to be either sub-gaussian or bounded. Since we allow for dependence across $t=1,...,T$ and only assume the existence of certain moments neither the independence nor the boundedness/sub-gaussianity is satisfied in our framework, hence calling for a different approach. Our approach relies on showing that the entries of $X'\epsilon$ all form martingales with respect to a certain filtration under the stated assumptions and then applying Rosenthal's inequality. A benefit from this strategy is that explicit optimal upper bounds on the constant $C_r$ are known from \cite{hitczenko90}. In the appendix we also give an expression for $D_r$ which only depends on $C_r$ and $D_r$. Of course, these upper bounds may be large since they have to guard against the worst case behavior of all martingales.

$\xi_{N,T}$ may be interpreted as the punishment on the convergence rate for not knowing the true model. Since $a_{N,T}$ will in general be chosen to be an increasing sequence one sees that in the setting of fixed $T, p, s_1$ and $s_2$ the upper bound on $\enVert[1]{\hat{\beta}-\beta^*}$ is of the order $a_{N,T}N^{1/r-1/2}$ (if $\kappa^2$ is bounded away from zero) which is not far from $1/\sqrt{N}$ if $r$ is large and $a_{N,T}$ is increasing slowly. As in Theorem \ref{thm1}, the bounds in (\ref{IQ1thm2}) and (\ref{IQ2thm2}) are valid uniformly over $\cbr[0]{\beta^*\in\mathbb{R}^p:\enVert{\beta^*}_{\ell_0}\leq s_1}\times \cbr[0]{c^*\in\mathbb{R}^n:\enVert{c^*}_{\ell_0}\leq s_2}$ underscoring that the only relevant characteristic of the true parameter vector which matters is its number of non-zero entries. 

If $\epsilon_{1,t}$ is uniformly bounded in $L_r$, which is the case if they are e.g. identically distributed, then the term $\max_{1\leq t\leq T}\enVert[0]{\epsilon_{1,t}}_{L_r}$ can be disregarded in asymptotic considerations. Furthermore, (\ref{IQ2thm2}) confirms the well known fact that $T$ must be large in order to estimate $c^*$ precisely since there are only $T$ observation per $c_i^*,\ i=1,...,N$.

We also stress that Theorem \ref{thm2} does not require sub-gaussianity of the covariates and the error terms and in this respect it relaxes one of the standard assumptions in the high-dimensional modeling literature. It also allows for the error terms to be heteroscedastic and dependent over time. In the jargon of the unobserved ability example from the introduction, Theorem \ref{thm2} allows us to estimate the effect of ability on wages since we know that $\hat{c}$ will not be too far away from $c^*$.
The next theorem is similar in spirit to Theorem \ref{thm2} but strengthens the existence of $r$ moments to sub-gaussian tails of the covariates as well as the error terms, i.e. we invoke A2b) instead of A2a).

\begin{theorem}\label{thm3}
Let assumption A1) and A2b) be satisfied and assume that $\kappa^2>0$. Then, choosing $\lambda_{N,T}=\sqrt{4NT\log(p)^3\log(a_{N,T})^3}$ and $\mu_{N,T}=\sqrt{4T\log(N)^3\log(a_{N,T})^3}$ for any sequence $a_{N,T}\geq e$ one has $P\del[0]{\mathcal{A}_{N,T}\cap \mathcal{B}_{N,T}}\geq 1-Ap^{1-B\log(a_{N,T})}-AN^{1-B\ln(a_{N,T})}-A(p^2+Np) e^{-B(t^2N)^{1/3}}$ for absolute constants $A$ and $B$, $t=\frac{\kappa^2}{(s_1+s_2)\del[1]{\frac{
\ln(p)}{\ln(N)}\vee \frac{\ln(N)}{\ln(p)}}^3}$ and $Nt^2\geq 1$. Furthermore, with at least this probability (i.e. on $\mathcal{A}_{N,T}\cap \mathcal{B}_{N,T}$),
\begin{align}
\enVert[1]{\hat{\beta}-\beta^*}
\leq
\frac{\xi_{N,T}}{\sqrt{NT}}\label{IQ1thm3}
\end{align}
and
\begin{align}
\enVert[1]{\hat{c}-c^*}
\leq 
\frac{\xi_{N,T}}{\sqrt{T}}\label{IQ2thm3}
\end{align}
where $\xi_{N,T}=16\log(a_{N,T})^{3/2}\sbr[1]{\log(p)^{3/2}\sqrt{s_1}+\log(N)^{3/2}\sqrt{s_2}}/\kappa^2$.
\end{theorem}
The form of the upper bounds on the estimation errors is the same as in Theorem \ref{thm2}. However, the definition of $\xi_{N,T}$ has changed. In particular, $\xi_{N,T}$ is now increasing slower in the number of variables, $p$, in $X$ and $N$ in $D$, respectively. In the case where $T, p, s_1$ and $s_2$ are bounded the upper bound on $\enVert[1]{\hat{\beta}-\beta^*}$ is of order $\ln(a_{N,T})^{3/2}\ln(N)^{3/2}/\sqrt{N}$ (if $\kappa^2$ is bounded away from 0). In other words, the punishment for not knowing the true model is now merely logarithmic in the sample size. As in the two previous theorems the upper bounds in (\ref{IQ1thm3}) and (\ref{IQ2thm3}) are valid uniformly over $\cbr[0]{\beta^*\in\mathbb{R}^p:\enVert{\beta^*}_{\ell_0}\leq s_1}\times \cbr[0]{c^*\in\mathbb{R}^n:\enVert{c^*}_{\ell_0}\leq s_2}$.

Readers familiar with the standard high-dimensional linear regression model with independent observations and sub-gaussian error terms would most likely expect $\lambda_{N,T}=\sqrt{4NT\log(p)\log(a_{N,T})^3}$ to be sufficiently large to provide a useful lower bound on the probability of $\mathcal{A}_{N,T}$. However, we pay a penalty of $\log(p)$ for the data not being independent over $t=1,...,T$. Instead, as under Assumption A2a), we proceed by showing that the entries of $X'\epsilon$ are martingales with respect to a suitably chosen filtration under the stated assumptions. This allows us to use a concentration inequality for unbounded martingales due to \cite{lesignev01} which is known to be optimal. Hence, unless an entirely different approach is taken, the extra $\log(p)$ term seems unavoidable in general. In any case this term is merely logarithmic in $p$.

So far, we have focussed on providing upper bounds on the estimation error. An obvious question is now how tight these bounds are. It turns out that the established bounds are indeed tight. In particular, we show next that no improvements can be made beyond multiplicative constants. First, note that Theorem \ref{thm3} implies that
\begin{align}
\enVert{S_T(\hat{\gamma}-\gamma^*)}\leq 2\xi_{N,T}\label{STbound}
\end{align}
with high probability\footnote{To be precise, with at least the lower bound on $P(\mathcal{A}_{N,T}\cap \mathcal{B}_{N,T})$ provided in Theorem \ref{thm3}.}. The following theorem shows that the upper bound in (\ref{STbound}) cannot be improved in the case of gaussian error terms.

\begin{theorem}\label{lowerbound}
Let A1) and A2b) be satisfied and assume that $\kappa^2$ is bounded away from zero. Assume that $\epsilon_i$ is $N(0,\sigma^2 I_T)$ and $\phi_{\min}(\Gamma_{J,J}), \kappa^2$ are bounded from below and $\phi_{\max}(\Gamma_{J,J})$ is bounded from above. Choose $\lambda_{N,T}$ and $\mu_{N,T}$ as in Theorem \ref{thm3}. Then when the Lasso detects the correct sparsity pattern, it holds with probability at least $1-\exp(-c_1|J|)-A(p^2+Np) e^{-B(t^2N)^{1/3}}$ that
\begin{align}
\enVert[1]{S_{J,J}(\hat{\gamma}_J-\gamma_J^*)}
\geq
c_2\xi_{N,T}\label{lower}
\end{align}
for absolute constants $c_1,\ c_2$, $A$ and $B$ and $t=\frac{\kappa^2}{(s_1+s_2)\del[1]{\frac{
\ln(p)}{\ln(N)}\vee \frac{\ln(N)}{\ln(p)}}^3}$ as long as $Nt^2\geq 1$ where $\xi_{N,T}$ is as in Theorem \ref{thm3}.
\end{theorem}
Inequality (\ref{lower}) is the reverse inequality of (\ref{STbound}) and shows that one cannot improve the bounds in Theorem \ref{thm3} except for multiplicative constants. Hence, our results are sharp and we turn next towards the asymptotic implications of our finite sample bounds.

\section{Asymptotic properties of the Lasso}\label{asymLasso}
In this section we show that the Lasso can estimate $\beta^*$ and $c^*$ consistently in even very high-dimensional settings where the number of covariates increases exponentially in $N$. It is also shown that no relevant variables will be discarded from the model as long as $\beta_{\min}$ and $c_{\min}$ do not tend to zero too fast. Let $a_{N,T}=N, T=N^a, p=e^{N^b}$ and $s_1=s_2=N^{c}$ for $a,b,c\geq 0$. Then we have the following result which builds upon Theorem \ref{thm3}.

\begin{theorem}\label{LassoAsym}
Let assumptions A1) and A2b) be satisfied and assume that $\kappa^2$ is bounded away from zero. Then, if $9b+2c< 1$ as $N\to \infty$ one has with probability tending to one
\begin{enumerate}
	\item 
\begin{align*}
&\enVert[1]{\hat{\beta}-\beta^*}\to 0 \\
&\enVert[1]{\hat{c}-c^*}\to 0 \text{ if } 3b+c<a
\end{align*}
\item $\hat{\beta}_j$ will not be classified as zero for any $j\in J_1$ if $\beta_{\min}>\xi_{N,T}/\sqrt{NT}$ from a certain step and onwards. Similarly, no $\hat{c}_i$ will be classified as zero for $i\in J_2$ if $c_{\min}>\xi_{N,T}/\sqrt{T}$ from a certain step and onwards.
\end{enumerate}
 
\end{theorem}

The first part of Theorem \ref{LassoAsym} shows that even when $p$ increases exponentially in $N$, it is possible for the Lasso to be consistent for $\beta^*$ as well as $c^*$. Put differently, the Lasso can be consistent in even ultra high-dimensional models. However, and as can be expected, one must have $a>0$ in order to estimate $c^*$ consistently since only $T=N^a$ observations are available to estimate each $c_i^*,\ i=1,...,N$. In the case of standard \textit{large N} asymptotics ($a=0$), the Lasso can still be consistent for $\beta^*$ as long as $9b+2c< 1$. This is clearly satisfied in the standard setting of fixed $p$, $s_1$ and $s_2$ ($b=c=0$).   

The second part of the theorem reveals that the Lasso can be used as a strong screening device since no relevant variables will be excluded from the model if their coefficients are not too close to zero. The necessity of such a "beta-min" (or "c-min") condition is not surprising since one cannot expect to be able to distinguish non-zero parameters from zero ones if the distance between these is too small. It is not difficult to see that in the standard large $N$ setting of $a=b=c=0$, the "beta-min" condition requires $\beta_{\min}\geq \log(N)^3/\sqrt{N}$. Hence, all non-zero parameters outside a disc centered at zero with radius $\log(N)^3/\sqrt{N}$ will also be classified as non-zero by the Lasso. In the same setting, the "c-min" condition requires $c_{\min}\geq \ln(N)^3$ implying that in the limit only $c_i\geq \ln(N)^3,\ i\in J_2$ can be guaranteed to be classified as non-zero. Put differently, only large $c^*_i$ can be guaranteed to be classified as non-zero. One must have $a>3b+c$ in order for this disc to have a radius which tends to zero, i.e. to make sure that any non-zero $c_i^*$ which is bounded away from zero will be classified as non-zero asymptotically. The necessity of the non-zero parameters being bounded away from zero is not surprising in the light of the work of \cite{potscherl09} who document some of the limitations of the Lasso-type estimators.

It is also worth mentioning that the conditions of Theorem \ref{LassoAsym} are merely sufficient. For example it is also possible to let $\kappa^2$ tend to zero at the price of slower growth rates in the other variables without sacrificing consistency. Furthermore, one could also use Theorem \ref{thm2} instead of Theorem \ref{thm3} to deduce a theorem in the spirit of Theorem \ref{LassoAsym}. Of course, the models sizes would no longer be allowed to increase as fast as above.

\section{The Adaptive Lasso}\label{ALasso}
So far we have focussed on deriving upper bounds on the estimation error that hold with high probability. Next, we turn to variable selection. Variable selection is important for policy makers since they might be interested in exactly those variables which influence the phenomenon they are modeling.

 The Lasso penalizes all parameters equally much. This implies that it can only recover the correct sparsity pattern under rather stringent assumptions. If one could penalize the truly zero parameters more than the non-zero ones, one would expect a better performance. This idea was utilized by \cite{zou06} to propose the adaptive Lasso in the standard linear regression model with a fixed number of non-random regressors. He established that the adaptive Lasso can detect the correct sparsity pattern \textit{asymptotically} in such a setting. This motivates us to modify the adaptive Lasso to make it applicable in the linear panel data model and to derive lower bounds on the finite sample probabilities with which it selects the correct sparsity pattern. The adaptive Lasso estimates $\beta^*$ and $c^*$ by minimizing the following objective function:

\begin{align}
\tilde{L}(\beta, c)
=
\sum_{i=1}^N\sum_{t=1}^T\del[1]{y_{i,t}-x_{i,t}'\beta-c_i}^2+\lambda_{N,T}\sum_{k\in J_1(\hat{\beta})}\frac{\envert[0]{\beta_k}}{\envert[0]{\hat{\beta}_k}}+\mu_{N,T}\sum_{i\in J_2(\hat{c})}\frac{\envert[0]{c_i}}{\envert[0]{\hat{c}_i}}\label{aLassoObj}
\end{align}
where $J_1(\hat{\beta})=\cbr[0]{j: \hat{\beta}_j\neq 0}$ and $J_2(\hat{c})=\cbr[0]{i:\hat{c}_i\neq 0}$. Denote the minimizers of $L$ by $\tilde{\beta}$ and $\tilde{c}$, respectively. Note that if $\hat{\beta}_j$ or $\hat{c}_i$ equal zero, the corresponding variable is entirely excluded from the model in the second step. Hence, the dimension of the second step estimation can be of a much smaller order of magnitude than the first step estimation. If $\beta^*_j=0$ then it follows by Theorems \ref{thm2} and \ref{thm3} that $\hat{\beta}_j$ is likely to be small (or even 0) and so the penalty on $\beta_j$ in (\ref{aLassoObj}) is large implying that $\tilde{\beta}$ is likely to be classified as being zero. The reverse logic applies when $\beta_j^*\neq 0$ (and similarly for $c_i^*$). Put differently, the adaptive Lasso is a two-step estimator which uses more intelligent weights than the ordinary Lasso. We shall see next, that these more intelligent weights imply that the adaptive Lasso can select the correct sparsity pattern. As for the Lasso, we start with a purely deterministic result to which we then attach probabilities by adding assumptions A1) and A2a) or A2b). First, define the sets
\begin{align*}
\mathcal{C}_{1,N,T}=\cbr[2]{\max_{k\in J_1^c}\max_{l\in J_1}\frac{1}{\sqrt{NT}}\sum_{i=1}^{N}\sum_{t=1}^Tx_{i,t,k}x_{i,t,l} \vee \max_{i\in J_2}\max_{k\in J_1^c}\frac{1}{\sqrt{T}}\sum_{t=1}^Tx_{i,t,k}\leq K_{1,N,T}
},\\
\mathcal{C}_{2,N,T}=\cbr[2]{\max_{i\in J_2^c}\max_{k\in J_1}\frac{1}{\sqrt{T}}\sum_{t=1}^Tx_{i,t,k}\leq K_{2,N,T}}
\text{ and }
\mathcal{D}_{N,T}=\cbr[1]{\phi_{\min}(\Psi_{J,J})\geq \phi_{\min}(\Gamma_{J,J})/2}.
\end{align*}
$\mathcal{C}_{1,N,T}$ may be interpreted as the set where none of the irrelevant $x_j$'s has a too big inner product (in $\ell_2$), or covariance, with any of the relevant $x_j$'s or dummies in $D$. Similarly $\mathcal{C}_{2,N,T}$ is the set where none of the relevant dummies is too highly correlated with any of the relevant $x_j$'s (all dummies are orthogonal in $\ell_2$ by construction so no condition is needed on their interdependence). On these sets, the problem is well-posed in the sense that the relevant and irrelevant variables are not too highly correlated and hence we can distinguish between them as we will see below. On the set $\mathcal{D}_{N,T}$, one basically has that $\Psi_{J,J}$ is bounded away from singularity. With these definitions in place we may state the following theorem.    

\begin{theorem}\label{aLassoThm1}
On $\mathcal{A}_{N,T}\cap\mathcal{C}_{1,N,T} \cap\mathcal{D}_{N,T}\cap\cbr[0]{\enVert[0]{\hat{\beta}-\beta^*}\leq \beta_{\min}/2}\cap\cbr[0]{\enVert[0]{\hat{c}-c^*}\leq c_{\min}/2}$ one has $\sgn(\tilde{\beta})=\sgn(\beta^*)$ if
\begin{align}
\frac{2\sqrt{|J|}}{\phi_{\min}(\Gamma_{J,J})}\del[3]{\frac{\lambda_{N,T}}{2\sqrt{NT}}\vee \frac{\mu_{N,T}}{2\sqrt{T}}+\frac{2\lambda_{N,T}}{\sqrt{NT}\beta_{\min}}\vee \frac{2\mu_{N,T}}{\sqrt{T}c_{\min}}}
&\leq
\sqrt{NT}\beta_{\min}\label{albeta1}\\
\frac{2|J|K_{1,N,T}}{\phi_{\min}(\Gamma_{J,J})}\del[3]{\frac{\lambda_{N,T}}{2\sqrt{NT}}\vee \frac{\mu_{N,T}}{2\sqrt{T}}+\frac{2\lambda_{N,T}}{\sqrt{NT}\beta_{\min}}\vee \frac{2\mu_{N,T}}{\sqrt{T}c_{\min}}}+\frac{\lambda_{N,T}}{2}
&\leq
\frac{\lambda_{N,T}}{\enVert[0]{\hat{\beta}-\beta^*}}\label{albeta2}.
\end{align}
Similarly, on $\mathcal{A}_{N,T}\cap\mathcal{C}_{2,N,T}\cap\mathcal{D}_{N,T}\cap\cbr[0]{\enVert[0]{\hat{\beta}-\beta^*}\leq \beta_{\min}/2}\cap\cbr[0]{\enVert[0]{\hat{c}-c^*}\leq c_{\min}/2}$
 one has $\sgn(\tilde{c})=\sgn(c^*)$ if
\begin{align}
\frac{2\sqrt{|J|}}{\phi_{\min}(\Gamma_{J,J})}\del[3]{\frac{\lambda_{N,T}}{2\sqrt{NT}}\vee \frac{\mu_{N,T}}{2\sqrt{T}}+\frac{2\lambda_{N,T}}{\sqrt{NT}\beta_{\min}}\vee \frac{2\mu_{N,T}}{\sqrt{T}c_{\min}}}
&\leq
\sqrt{T}c_{\min}\label{alc1}\\
\frac{2|J|K_{2,N,T}}{\phi_{\min}(\Gamma_{J,J})}\del[3]{\frac{\lambda_{N,T}}{2\sqrt{NT}}\vee \frac{\mu_{N,T}}{2\sqrt{T}}+\frac{2\lambda_{N,T}}{\sqrt{NT}\beta_{\min}}\vee \frac{2\mu_{N,T}}{\sqrt{T}c_{\min}}}+\frac{\mu_{N,T}}{2}
&\leq
\frac{\mu_{N,T}}{\enVert[0]{\hat{c}-c^*}}\label{alc2}.
\end{align}
\end{theorem}
Note that, just as Theorem \ref{thm1}, Theorem \ref{aLassoThm1} is purely deterministic. Inequality (\ref{albeta1}) is sufficient to ensure that no relevant $x_j$'s are excluded from the model. It is sensible that the smaller $\beta_{\min}$ is the more difficult it is to avoid excluding relevant variables. This is reflected in (\ref{albeta1}) in that the left hand side is deceasing in $\beta_{\min}$ while the right hand side is increasing. Larger $\lambda_{N,T}$ and $\mu_{N,T}$ also make it harder to satisfy the inequality since too much shrinkage can result in relevant variables being discarded. On the other hand, as in Theorem \ref{thm1}, the size of $\mathcal{A}_{N,T}$ is increasing in these two quantities revealing the same tradeoff as discussed previously.

Inequality (\ref{albeta2}) gives a sufficient condition for not classifying any irrelevant $x_j$s as relevant. Note that the more precise the initial Lasso estimator is the larger is the right hand side and hence the more likely it is that the inequality is satisfied. Increasing $K_{1,N,T}$ allows for larger dependence between relevant and irrelevant variables and thus makes it harder to distinguish between these. Hence, it is sensible that the left hand side of (\ref{albeta2}) is increasing in $K_{1,N,T}$. On the other hand, the size of $\mathcal{C}_{1,N,T}$ is increasing in $K_{1,N,T}$. The intuition behind inequalities (\ref{alc1}) and (\ref{alc2}) is the same for the preceding two inequalities. At this point it is also worth mentioning that Theorem \ref{aLassoThm1} does not assume the use of the Lasso as initial estimator. The estimators $\hat{\beta}$ and $\hat{c}$ could be any estimators for which an upper bound on the estimation error is available and -- as can be seen -- more precise initial estimators will make the conditions of Theorem \ref{aLassoThm1} more likely to be satisfied. 

Next, we use the above theorem to give lower bounds on the probability with which the adaptive Lasso selects the correct sparsity pattern by invoking assumptions A1) and A2a) or A2b), respectively.  

\begin{corollary}\label{ALcor}
\begin{enumerate}
\item Let assumptions A1 and A2a) be satisfied and assume that (\ref{albeta1})-(\ref{albeta2}) are valid with $\lambda_{N,T}$ and $\mu_{N,T}$ as in Theorem \ref{thm2} and $K_{1,N,T}=|J_1^c|^{2/r}|J_1|^{2/r}(NT)^{1/2}a_{N,T}$. Assume that $\beta_{\min}\geq 2\frac{\xi_{N,T}}{\sqrt{NT}}$ and $c_{\min}\geq 2\frac{\xi_{N,T}}{\sqrt{T}}$ with $\xi_{N,T}$ as in Theorem \ref{thm2}. Then, $\sgn(\tilde{\beta})=\sgn(\beta^*)$ with probability at least $1-2\del{\frac{C_r}{a_{N,T}}}^r-D_r\frac{(p^2+Np)(s_1+s_2)^{r/2}\del[0]{\frac{p}{N}\vee \frac{N}{p}}}{\kappa^rN^{r/4}}-\frac{2}{a_{N,T}^{r/2}}$ for constants $C_r$ and $D_r$ only depending on $r$. Similarly, if (\ref{alc1})-(\ref{alc2}) are valid with $K_{2,N,T}=|J_1|^{1/r}|J_2^c|^{1/r}T^{1/2}a_{N,T}$ then $\sgn(\tilde{c})=\sgn(c^*)$ with probability at least $1-2\del{\frac{C_r}{a_{N,T}}}^r-D_r\frac{(p^2+Np)(s_1+s_2)^{r/2}\del[0]{\frac{p}{N}\vee \frac{N}{p}}}{\kappa^rN^{r/4}}-\frac{1}{a_{N,T}^{r/2}}$.
\item Let assumptions A1 and A2b) be satisfied and assume that (\ref{albeta1})-(\ref{albeta2}) are valid with $\lambda_{N,T}$ and $\mu_{N,T}$ as in Theorem \ref{thm3} and $K_{1,N,T}=A\log(1+|J_1^c|)\log(e+|J_1|)\sqrt{NT}\log(a_{N,T})$ for $A>0$. Assume that $\beta_{\min}\geq 2\frac{\xi_{N,T}}{\sqrt{NT}}$ and $c_{\min}\geq 2\frac{\xi_{N,T}}{\sqrt{T}}$ with $\xi_{N,T}$ as in Theorem \ref{thm3}. Then, $\sgn(\tilde{\beta})=\sgn(\beta^*)$ with probability at least $1-Ap^{1-B\log(a_{N,T})}-AN^{1-B\ln(a_{N,T})}-A(p^2+Np)e^{-B(t^2N)^{1/3}}-\frac{4}{a_{N,T}}$ for absolute constants $A$ and $B$ and $t=\frac{\kappa^2}{(s_1+s_2)\del[1]{\frac{
\ln(p)}{\ln(N)}\vee \frac{\ln(N)}{\ln(p)}}^3}$ as long as $Nt^2\geq 1$. Similarly, if (\ref{alc1})-(\ref{alc2}) are valid with $K_{2,N,T}=A\log(1+|J_1|)\log(1+|J_2^c|)\sqrt{T}\log(a_{N,T})$ then $\sgn(\tilde{c})=\sgn(c^*)$ with probability at least $1-Ap^{1-B\log(a_{N,T})}-AN^{1-B\ln(a_{N,T})}-A(p^2+Np)e^{-B(t^2N)^{1/3}}-\frac{2}{a_{N,T}}$.     
\end{enumerate}
\end{corollary}
Corollary \ref{ALcor} gives lower bounds on the probability with which the adaptive Lasso detects the correct sparsity pattern under the two sets of assumptions employed in Theorems \ref{thm2} and \ref{thm3}, respectively. Hence, returning once again to the ability example from the introduction, the adaptive Lasso is able to detect the group of workers whose unobserved ability influences their wages ($c_i^*\neq 0$). By a union bound Corollary \ref{ALcor} can also be used to derive a crude lower bound on $P(\sgn(\tilde{\beta})=\sgn(\beta^*), \sgn(\tilde{c})=\sgn(c^*))$. A tighter bound can be derived by optimizing the proof slightly. 

In order to get a feeling for the size of the models that the adaptive Lasso can detect the correct sparsity pattern in, we shall use part (2) of the Corollary \ref{ALcor} to establish the following asymptotic result. As in Theorem \ref{LassoAsym} we shall consider the asymptotic setting where $a_{N,T}=N, T=N^a, p=e^{N^b}$ and $s_1=s_2=N^{c}$ for $a,b,c\geq 0$.

\begin{theorem}\label{aLassoAsym}
Let assumptions A1 and A2b) be satisfied and let $\kappa$,  $\beta_{\min}$ and $c_{\min}$ be bounded away from 0. Assume furthermore, that $9b+2c< 1$. Then,
\begin{enumerate}
	\item $P\del[1]{\sgn(\tilde{\beta})=\sgn(\beta^*)}\to 1$ if $5b+3c<1+a$
	\item $P\del[1]{\sgn(\tilde{c})=\sgn(c^*)}\to 1$ if $6b+3c<a$.
\end{enumerate}
\end{theorem}
Part one of Theorem \ref{aLassoAsym} reveals that $p$ may increase at a sub-exponential rate while the number of relevant variables must increase slower than the square root of the sample size (set $b=0$ in $9b+2c<1$ to conclude that $c<1/2$) if the adaptive Lasso is to detect the correct sparsity pattern asymptotically. Actually, for $a<1/2$ the number of relevant variables must increase even slower. It is also worth noticing that sign consistency can be achieved in a fixed $T$ setting ($a=0$). This is in opposition to part 2 of the theorem: for the adaptive Lasso to be sign consistent for $c^*$ one needs $a>0$. This is of course sensible in the light of Theorem \ref{LassoAsym} since $a>0$ is needed for the first step Lasso estimator $\hat{c}$ to be consistent for $c^*$.

\section{Monte Carlo}\label{MC}
In this section we investigate the finite sample properties of the Lasso as well as the adaptive Lasso by means of Monte Carlo experiments. The Lasso is implemented using the publicly available \texttt{glmnet} package for R. Since $\mu_{N,T}/\lambda_{N,T}$ is roughly equal to $1/\sqrt{N}$ in Theorems \ref{thm2} and \ref{thm3} we can reduce the optimization problem to a search over only one tuning parameter in the following way:
\begin{enumerate}
\item Define $\tilde{D}=\sqrt{N}D$.
\item Minimize $\enVert[1]{y-X\beta-\tilde{D}c}^2+\lambda_{N,T}\sum_{k=1}^p\envert[0]{\beta_k}+\lambda_{N,T}\sum_{i=1}^N\envert[0]{c_i}$ wrt. $(\beta, c)$ by \texttt{glmnet} and denote the minimizer by $(\hat{\beta}, \hat{\hat{c}})$.
\item Return $(\hat{\beta}, \hat{c})=(\hat{\beta}, \sqrt{N}\hat{\hat{c}})$.
\end{enumerate} 

In step 2 above $\lambda_{N,T}$ is chosen by BIC. It is our experience that more time consuming procedures such as cross validation do not improve the results. The adaptive Lasso is implemented in the following way:
\begin{enumerate}
\item Define $\tilde{x}_j=x_{j}\hat{\beta}_j\ j=1,...,p$ and $\tilde{d}_i=\sqrt{N}\hat{c}_id_i,\ i=1,...,N$.
\item Minimize $\enVert[1]{y-\sum_{j=1}^p\tilde{x}_j\beta-\sum_{i=1}^N\tilde{d}_ic_i}^2+\lambda_{N,T}\sum_{k=1}^p\envert[0]{\beta_k}+\lambda_{N,T}\sum_{i=1}^N\envert[0]{c_i}$ wrt. $(\beta, c)$ by \texttt{glmnet} and denote the minimizer by $(\tilde{\tilde{\beta}}, \tilde{\tilde{c}})$.
\item Return $\tilde{\beta}_j=\tilde{\tilde{\beta}}_j\hat{\beta}_j, j=1,...,p$ and $\tilde{c}_i=\sqrt{N}\hat{c}_i\tilde{\tilde{c}}_i$.
\end{enumerate}
As for the Lasso, $\lambda_{N,T}$ is chosen by BIC. The above implementation of  the adaptive Lasso is similar in spirit to the one described in \cite{zou06}. To provide a benchmark for the Lasso and the adaptive Lasso, least squares including all variables is also implemented whenever feasible. This procedure is denoted OLSA. At the other extreme, least squares \textit{only} including the relevant variables is applied to provide an infeasible target which we are ideally aiming at. This procedure is called the OLS Oracle (OLSO). We measure the performance of the proposed estimators along the following dimensions

\begin{enumerate}
\item The average root mean square error of the parameter estimates of $\beta^*$ and $c^*$, i.e. the average $\ell_2$-estimation error.
\item How often is the true model included in the model chosen. This is relevant since even if the true model is not selected a good procedure should not exclude too many relevant variables. This measure is reported for $\beta^*$ as well as $c^*$.
\item How often is the correct sparsity pattern uncovered, i.e. how often is exactly the correct model chosen. This measure is reported for $\beta^*$ as well as $c^*$.
\item What is the mean number of non-zero parameters in the estimated model. This measures how much the dimension of the model is reduced and is reported for $\beta^*$ as well as $c^*$.
\end{enumerate}

The following experiments are carried out to gauge the performance along the above dimensions (the number of Monte Carlo replications is always 1000).
\begin{itemize}
\item Experiment A: N=T=10 with $\beta^*$ having five entries of 1 and 20 of zero. The non-zero entries are equidistant. $c^*$ has $floor(N^{1/3})=2$ entries of 1 and the rest zeros. The correlation between the $i$th and $j$th column of $X$ is $0.75^{|i-j|}$ and the covariates and error terms possess two moments only. 
\item Experiment B: As experiment A but with $N=100$ and $c^*$ having  $floor(N^{1/3})=4$. 
\item Experiment C: As experiment A but with $T=100$.
\item Experiment D: As experiment A but with gaussian covariates.
\item Experiment E: As experiment B but with gaussian covariates.	
\item Experiment F: As experiment C but with gaussian covariates.
\item Experiment G: As experiment A but now $\beta^*$ has five entries of one and 245 entries of zero. The non-zero entries are equidistant.
\item Experiment H: As experiment G but with gaussian covariates and error terms.
\item Experiment I:  N=T=10 with $\beta^*$ having 10 entries of 1 and 490 of zero. The non-zero entries are equidistant. $c^*$ has $floor(N^{1/3})=2$ entries of 1 and the rest zeros. The correlation between the $i$th and $j$th column of $X$ is $0.75^{|i-j|}$ and the covariates and error terms are gaussian. 
\end{itemize}

Experiments A-C are meant to illustrate Theorem \ref{thm2} and part 1 of Corollary \ref{ALcor}. Note that tails of the covariates and the error terms are extremely heavy in these experiments since they merely allow for the existence of two moments. Similarly, Experiments D-F are meant to illustrate Theorem \ref{thm3} and part 2 of Corollary \ref{ALcor} as the tails of the covariates and error terms are now subgaussian (in fact they are exactly gaussian) allowing the existence of all (polynomial) moments. Experiments G-H intend to investigate the performance of the Lasso and the adaptive Lasso in settings with more variables than observations and various moment assumptions on the covariates and the error terms. 

\subsection{Results}

\begin{table}[ht]
\centering
\begin{tabular}{cccccccccc}
  \toprule
& & MSE($\beta$) & MSE($c$) & Sub($\beta$) & Sub($c$) & Spar($\beta$) & Spar($c$) & $\#\beta$ & $\#c$ \\ 
  \midrule
\multirow{4}{*}{\begin{sideways}Exp A\end{sideways}} 
& Lasso & 1.02 & 1.16 & 0.87 & 0.42 & 0.01 & 0.09 & 9.17 & 2.47 \\ 
 & ALasso & 0.87 & 1.31 & 0.75 & 0.34 & 0.31 & 0.13 & 5.90 & 1.84 \\ 
 & OLSO & 0.39 & 0.64 & 1.00 & 1.00 & 1.00 & 1.00 & 5.00 & 2.00 \\ 
 & OLSA & 2.04 & 1.89 & 1.00 & 1.00 & 0.00 & 0.00 & 25.00 & 10.00 \\ 
\midrule
\multirow{4}{*}{\begin{sideways}Exp B\end{sideways}} 
 & Lasso & 0.33 & 2.15 & 1.00 & 0.00 & 0.04 & 0.00 & 8.06 & 2.26 \\ 
 & ALasso & 0.14 & 2.83 & 1.00 & 0.00 & 0.89 & 0.00 & 5.12 & 2.02 \\ 
 & OLSO & 0.12 & 0.97 & 1.00 & 1.00 & 1.00 & 1.00 & 5.00 & 4.00 \\ 
 & OLSA & 0.54 & 5.45 & 1.00 & 1.00 & 0.00 & 0.00 & 25.00 & 100.00 \\  
\midrule
\multirow{4}{*}{\begin{sideways}Exp C\end{sideways}} 
 & Lasso & 0.29 & 0.43 & 1.00 & 0.99 & 0.02 & 0.49 & 9.08 & 2.72 \\ 
 & ALasso & 0.14 & 0.28 & 1.00 & 0.99 & 0.91 & 0.84 & 5.11 & 2.16 \\ 
 & OLSO & 0.12 & 0.22 & 1.00 & 1.00 & 1.00 & 1.00 & 5.00 & 2.00 \\ 
 & OLSA & 0.51 & 0.54 & 1.00 & 1.00 & 0.00 & 0.00 & 25.00 & 10.00 \\ 
   \bottomrule
\end{tabular}
\caption{\small MSE($\beta$) and MSE($c$) are the average root mean square errors of the parameter estimates. Sub($\beta$) and Sub($c$) indicate the fraction of times the estimated model contains all the relevant variables (in $X$ and $D$) while Spar($\beta$) and Spar($c$) show how often exactly the correct subset of variables is chosen. Finally, $\#\beta$ and $\# c$ give the average number of non-zero $\beta$s and $c$s, respectively.}
\end{table}
 	 
Experiment A reveals that the Lasso as well as the adaptive Lasso estimate $\beta^*$ and $c^*$ at a precision which lies in between the one of least squares including all variables and the least squares oracle. The adaptive Lasso retains all non-zero $\beta^*$s in $75\%$ of the instances while only including 5.9 variables on average (recall that there are 5 relevant variables).

Increasing $N$ to 100, Experiment B shows that $\beta^*$ is now estimated more precisely while the opposite is the case for $c^*$. It is to be expected, however, that the mean square error of $\hat{c}$ increases since the vector now has 100 entries to be estimated as opposed to only 10 in Experiment A. The adaptive Lasso always retains all non-zero $\beta^*$s while detecting exactly the right sparsity pattern in $89\%$ of the cases. This is never the case for $c^*$, the reason being the same as mentioned above.

In Experiment C, $T$ is increased to 100 while $N=10$. This results in a higher precision of all estimators. In particular, the adaptive Lasso estimates $\beta^*$ and $c^*$ almost as precisely as the least squares oracle. The number of selected variables is also close to the ideal number.

Experiments D-F use gaussian covariates and error terms instead of ones with only two moments. Comparing the results to those in Experiments A-C reveals that the Lasso and the adaptive Lasso perform better now. Note for example, in Experiment D, the adaptive Lasso does not estimate $\beta^*$ much less precisely than the least squares oracle while in Experiment A it was more than twice as imprecise. Furthermore, all non-zero $c^*$ are classified as such by the Lasso in $81\%$ of the Monte Carlo replications while in the corresponding number in Experiment A was only $42\%$.

Moving from Experiment D to E all measures pertaining to $\beta^*$ improve -- the parameters are estimated more precisely (the adaptive Lasso is actually as precise as the least squares oracle) and the correct sparsity pattern is selected more than 9 out of ten times. As can be expected all measures pertaining to $c^*$ worsen since the number of parameters to be estimated ten-doubles.

In Experiment F, the Lasso and the adaptive Lasso perform well along all dimensions.

\begin{table}[ht]
\centering
\begin{tabular}{cccccccccc}
  \toprule
& & MSE($\beta$) & MSE($c$) & Sub($\beta$) & Sub($c$) & Spar($\beta$) & Spar($c$) & $\#\beta$ & $\# c$ \\ 
  \midrule
\multirow{4}{*}{\begin{sideways}Exp D\end{sideways}} 
 & Lasso & 0.57 & 0.78 & 1.00 & 0.81 & 0.01 & 0.24 & 9.56 & 3.14 \\ 
 & ALasso & 0.34 & 0.72 & 1.00 & 0.74 & 0.62 & 0.41 & 5.55 & 2.30 \\ 
 & OLSO & 0.23 & 0.41 & 1.00 & 1.00 & 1.00 & 1.00 & 5.00 & 2.00 \\ 
 & OLSA & 1.14 & 1.14 & 1.00 & 1.00 & 0.00 & 0.00 & 25.00 & 10.00 \\ 
\midrule
\multirow{4}{*}{\begin{sideways}Exp E\end{sideways}} 
& Lasso & 0.19 & 1.55 & 1.00 & 0.26 & 0.05 & 0.06 & 8.17 & 3.63 \\ 
 & ALasso & 0.08 & 1.40 & 1.00 & 0.23 & 0.93 & 0.09 & 5.08 & 3.17 \\ 
 & OLSO & 0.07 & 0.59 & 1.00 & 1.00 & 1.00 & 1.00 & 5.00 & 4.00 \\ 
 & OLSA & 0.31 & 3.20 & 1.00 & 1.00 & 0.00 & 0.00 & 25.00 & 100.00 \\ 
\midrule
\multirow{4}{*}{\begin{sideways}Exp F\end{sideways}} 
& Lasso & 0.17 & 0.25 & 1.00 & 1.00 & 0.02 & 0.56 & 9.05 & 2.65 \\ 
 & ALasso & 0.07 & 0.14 & 1.00 & 1.00 & 0.96 & 0.92 & 5.04 & 2.08 \\ 
 & OLSO & 0.07 & 0.12 & 1.00 & 1.00 & 1.00 & 1.00 & 5.00 & 2.00 \\ 
 & OLSA & 0.29 & 0.31 & 1.00 & 1.00 & 0.00 & 0.00 & 25.00 & 10.00 \\
\bottomrule
\end{tabular}
\caption{\small MSE($\beta$) and MSE($c$) are the average root mean square errors of the parameter estimates. Sub($\beta$) and Sub($c$) indicate the fraction of times the estimated model contains all the relevant variables (in $X$ and $D$) while Spar($\beta$) and Spar($c$) show how often exactly the correct subset of variables is chosen. Finally, $\#\beta$ and $\# c$ give the average number of non-zero $\beta$s and $c$s, respectively.}
\end{table}

Experiments G-H are the truly high-dimensional ones where the number of variables is (much) larger than the sample size. Hence, we do not implement least squares using all variables. Experiment G illustrates a rather difficult setting with many heavy-tailed covariates. The Lasso does a decent job in reducing dimensionality without being overwhelming either. The average number of non-zero $\hat{\beta}$s is 36.97 which is still larger than the five true non-zero coefficients. The adaptive Lasso removes ten more variables without discarding (many) more relevant ones so the second step seems worth implementing. 

In Experiment H the covariates are gaussian and the Lasso and the adaptive Lasso perform much better than in the heavy-tailed Experiment G. The estimation error of $\hat{\beta}$ is more than halved compared to Experiment G and all relevant variables are retained in the model. This does not come at the price of bigger models since the average number of non-zero coefficients is now smaller than before. The adaptive Lasso only classifies 17.64 $\beta$s as zero (of which five are truly non-zero) resulting in a significant dimension reduction. 

Experiment I doubles the number of variables in $X$ compared to Experiment H. In this light, it is reasonable that the estimation error of $\hat{\beta}$ roughly doubles. Almost all non-zero $\beta^*$ are also classified as such but unfortunately, though not unexpectedly, the total number of $\beta$s classified as non-zero also roughly doubles. However, the adaptive Lasso still manages to reduce the number of variables to less than one tenth of the original number of variables. 

\begin{table}[ht]
\centering
\begin{tabular}{cccccccccc}
\toprule
& & MSE($\beta$) & MSE($c$) & Sub($\beta$) & Sub($c$) & Spar($\beta$) & Spar($c$) & $\#\beta$ & $\# c$ \\ 
\midrule
\multirow{4}{*}{\begin{sideways}Exp G\end{sideways}} 
 & Lasso & 1.73 & 1.42 & 0.67 & 0.30 & 0.00 & 0.05 & 36.97 & 2.77 \\ 
 & ALasso & 1.66 & 1.51 & 0.62 & 0.26 & 0.05 & 0.05 & 26.93 & 2.45 \\ 
 & OLSO & 0.37 & 0.67 & 1.00 & 1.00 & 1.00 & 1.00 & 5.00 & 2.00 \\ 
 & OLSA &  &  &  &  &  &  &  &  \\\midrule
\multirow{4}{*}{\begin{sideways}Exp H\end{sideways}} 
& Lasso & 0.87 & 1.05 & 1.00 & 0.49 & 0.01 & 0.24 & 24.46 & 2.28 \\ 
 & ALasso & 0.66 & 0.94 & 1.00 & 0.48 & 0.20 & 0.25 & 17.54 & 2.13 \\ 
 & OLSO & 0.22 & 0.40 & 1.00 & 1.00 & 1.00 & 1.00 & 5.00 & 2.00 \\ 
 & OLSA &  &  &  &  &  &  &  &  \\
\midrule
\multirow{4}{*}{\begin{sideways}Exp I\end{sideways}} 
& Lasso & 1.43 & 1.03 & 0.97 & 0.55 & 0.00 & 0.14 & 63.95 & 2.90 \\ 
 & ALasso & 1.20 & 0.99 & 0.93 & 0.49 & 0.04 & 0.17 & 38.27 & 2.43 \\ 
 & OLSO & 0.33 & 0.43 & 1.00 & 1.00 & 1.00 & 1.00 & 10.00 & 2.00 \\ 
 & OLSA &  &  &  &  &  &  &  &  \\ 
\bottomrule
\end{tabular}
\caption{\small MSE($\beta$) and MSE($c$) are the average root mean square errors of the parameter estimates. Sub($\beta$) and Sub($c$) indicate the fraction of times the estimated model contains all the relevant variables (in $X$ and $D$) while Spar($\beta$) and Spar($c$) show how often exactly the correct subset of variables is chosen. Finally, $\#\beta$ and $\# c$ give the average number of non-zero $\beta$s and $c$s, respectively.}
\end{table}

\section{Empirical illustration}\label{Emp}
In this section we illustrate the use of the panel (adaptive) Lasso on a large data set for the G8 countries. In particular, we try to determine which variables are relevant for explaining economic growth in these countries. The neoclassical growth model predicts that higher initial wealth should lead to lower growth rates. The primary mechanism behind this prediction is that countries with low capital to labor ratios tend to have a higher marginal return to capital, \cite{barro91}. In this section we shall investigate whether this prediction is true for some of the biggest economies in the world.

The data set has been obtained from the data bank of world development indicators. The panel that we analyse consists of 8 countries with 20 annual observations for each country for the period 1992-2011. The number of explanatory variables (excluding the eight individual effects dummies) is 161. Hence, the number of variables is large compared to the number of observations and the Lasso-type estimators come to use since they offer a non ad hoc way of choosing the variables. Put differently, one can handle a much larger conditioning set of variables than previous methods.

The variables cover broad categories such as economical, health, demographical and technological ones. The GDP level is treated specially in the sense that it enters the right hand side of the model with a lag of one year to enable us to test whether initial GDP is related (negatively) to GDP growth. All right hand side variables are standardised to have an $\ell_2$-norm equal to the sample size. The Lasso as well as the adaptive Lasso are implemented by the \texttt{glmnet} as in the Monte Carlo section.

Table \ref{tab4} contains the results of the estimation. In the first round $\lambda_{N,T}$ is chosen by BIC for the Lasso as well as the adaptive Lasso. Then it is gradually reduced by choosing decreasing fractions of this initial choice. This is done as a kind of sensitivity check to verify the robustness of the sparsity. As can be seen from Table \ref{tab4} the Lasso and the adaptive Lasso indeed choose very sparse models when $\lambda_{N,T}$ is chosen by BIC. In particular, they include three and two variables, respectively. Note that all variables chosen are annual growth rates. This is sensible since we are trying to explain the \textit{annual growth rate} of GDP. Furthermore, it is seen that initial GDP does not enter as an explanatory variable. Hence, we find no support for the neoclassical growth hypothesis. However, it should be said that this hypothesis might be more relevant at explaining differences in growth between developed and less developed countries while all countries in our sample are rather developed. Furthermore, we use the GDP of the previous year as initial GDP which is a choice that might not leave enough time for the transmission mechanisms to function properly.

As can be expected, lowering $\lambda_{N,T}$ results in more variables being included in the model. This is manifested in Table \ref{tab4} by the models becoming gradually larger as $\lambda_{N,T}$ is decreased. But only for $\lambda_{N,T}=0.1\cdot \lambda_{BIC}$ a dummy is included in the model by the Lasso (for the United Kingdom).

\begin{landscape}
\begin{table}
\scalebox{.75}{
\begin{tabular}{ccc}
\toprule
$\lambda_{N,T}$ & Lasso & Adaptive Lasso\\
\midrule
$\lambda_{BIC}$ &Exports of goods and services (annual \% growth)& \\                   
&Gross fixed capital formation (annual \% growth)&Gross fixed capital formation (annual \% growth)\\                  
&Household final consumption expenditure (annual \% growth)&Household final consumption expenditure (annual \% growth) \\
\midrule
$0.75\cdot\lambda_{BIC}$&Exports of goods and services (annual \% growth)&Exports of goods and services (annual \% growth)\\                   
&General government final consumption expenditure (annual \% growth)& \\
&Gross fixed capital formation (annual \% growth)&Gross fixed capital formation (annual \% growth)\\                  
&Household final consumption expenditure (annual \% growth)&Household final consumption expenditure (annual \% growth) \\
\midrule
$0.5\cdot\lambda_{BIC}$&Exports of goods and services (annual \% growth)&Exports of goods and services (annual \% growth)\\                   
&General government final consumption expenditure (annual \% growth)& \\
&Gross fixed capital formation (annual \% growth)&Gross fixed capital formation (annual \% growth)\\                  
&Household final consumption expenditure (annual \% growth)&Household final consumption expenditure (annual \% growth) \\
&Market capitalization of listed companies (\% of GDP)& \\
\midrule
$0.25\cdot\lambda_{BIC}$&Exports of goods and services (annual \% growth)&Exports of goods and services (annual \% growth)\\                   
&General government final consumption expenditure (annual \% growth)&General government final consumption expenditure (annual \% growth)\\
&Gross fixed capital formation (annual \% growth)&Gross fixed capital formation (annual \% growth)\\                  
&Household final consumption expenditure (annual \% growth)&Household final consumption expenditure (annual \% growth) \\
&Market capitalization of listed companies (\% of GDP)& \\
\midrule
$0.1\cdot\lambda_{BIC}$ &Exports of goods and services (annual \% growth)& Exports of goods and services (annual \% growth)                \\            
&Final consumption expenditure, etc. (annual \% growth)&   Final consumption expenditure, etc. (annual \% growth)       \\            
&Forest rents (\% of GDP)&                                                 \\
&General government final consumption expenditure (annual \% growth)& General government final consumption expenditure (annual \% growth)         \\
&Gross fixed capital formation (annual \% growth)&    Gross fixed capital formation (annual \% growth)                        \\
&Household final consumption expenditure (annual \% growth)&         \\      
&Household final consumption expenditure, PPP (constant 2005 international \$)& \\
&Inflation, GDP deflator (annual \%)&  Inflation, GDP deflator (annual \%)                       \\ 
&Market capitalization of listed companies (\% of GDP)&            \\           
& UK dummy & \\
\bottomrule
\end{tabular}
}
\caption{\small The table shows which variables were chosen by the Lasso and the adaptive Lasso for various choices of $\lambda_{N,T}$. $\lambda_{BIC}$ indicates that $\lambda_{N,T}$ was chosen by BIC while the other sections of the table contain the results for $\lambda_{N,T}$ being a certain fraction of $\lambda_{BIC}$.} 
\label{tab4}
\end{table}
\end{landscape}

\section{Conclusion}\label{concl}
High-dimensional data is becoming increasingly available and one of the first choices one has to make when building a model is which variables to include. Furthermore, panel data models are a work horse tool for microeconometric analysis. For these reasons we have studied the performance of the panel Lasso and adaptive Lasso in high-dimensional panel data models.
In particular, this paper has established finite sample upper bounds on the estimation error of the panel Lasso estimator that hold with high probability. We have also shown that the upper bounds are optimal in a sense made clear in Theorem \ref{lowerbound}. Conditions for consistency in even very high-dimensional models were also provided. 

Next, the panel adaptive Lasso was analyzed and we gave lower bounds on the probability with which it selects the correct sign pattern in finite samples. These results were then used to deduce asymptotic results. 

The results were proven under various assumptions on the moment/tail behavior of the covariates and the error terms. In particular, we allowed for non-subgaussian behavior in some of our theorems.

The methods were then applied to finding the variables that explain growth in the G8 countries over the last 20 years. A rather sparse model was found to explain the growth.

In this paper we have used BIC to select the tuning parameters but ideally one would like a data driven way with theoretical guarantees. We leave this as an interesting avenue for future research. Furthermore, we have considered static panel date models in this paper. We conjecture that the results will carry over to dynamic panel data models by combining the present techniques with the ones we have previously used in a purely dynamic setting in \cite{kock2013oracle}.

\section{Appendix}
We start with the following Lemma which is similar in spirit to Lemma B.1 in \cite{bickelrt09}.

\begin{lemma}\label{masterlemma}
On $\mathcal{A}_{N,T}$ the following inequalities are valid.
\begin{align}
\enVert[1]{Z(\hat{\gamma}-\gamma^*)}^2+\lambda_{N,T}\enVert[1]{\hat{\beta}-\beta^*}_{\ell_1}+\mu_{N,T}\enVert[1]{\hat{c}-c^*}_{\ell_1}
\leq
4\lambda_{N,T}\enVert[1]{\hat{\beta}_{J_1}-\beta^*_{J_1}}_{\ell_1}+4\mu_{N,T}\enVert[1]{\hat{c}_{J_2}-c^*_{J_2}}_{\ell_1}\label{IQ1}
\end{align}
and
\begin{align}
\lambda_{N,T}\enVert[1]{\hat{\beta}_{J_1^c}-\beta^*_{J_1^c}}_{\ell_1}+\mu_{N,T}\enVert[1]{\hat{c}_{J_2^c}-c^*_{J_2^c}}_{\ell_1}
\leq
3\lambda_{N,T}\enVert[1]{\hat{\beta}_{J_1}-\beta^*_{J_1}}_{\ell_1}+3\mu_{N,T}\enVert[1]{\hat{c}_{J_2}-c^*_{J_2}}_{\ell_1}\label{IQ2}
\end{align}
\end{lemma}
\begin{proof}
By the minimizing property of $\hat{\gamma}$ it follows that
\begin{align*}
\enVert[1]{y-Z\hat{\gamma}}^2+2\lambda_{N,T}\enVert[1]{\hat{\beta}}_{\ell_1}+2\mu_{N,T}\enVert[0]{\hat{c}}_{\ell_1}
\leq 
\enVert[1]{y-Z\gamma^*}^2+2\mu_{N,T}\enVert[1]{\beta^*}_{\ell_1}+2\mu_{N,T}\enVert[0]{c^*}_{\ell_1}
\end{align*}
which, using that $y=Z\gamma^*+\epsilon$, yields
\begin{align*}
\enVert[1]{Z(\hat{\gamma}-\gamma^*)}^2-2\epsilon'Z(\hat{\gamma}-\gamma^*)+2\lambda_{N,T}\enVert[1]{\hat{\beta}}_{\ell_1}+2\mu_{N,T}\enVert[0]{\hat{c}}_{\ell_1}
\leq 
2\lambda_T\enVert[1]{\beta^*}_{\ell_1}+2\mu_{N,T}\enVert[0]{c^*}_{\ell_1}
\end{align*}
Or, equivalently
\begin{align}
\enVert[1]{Z(\hat{\gamma}-\gamma^*)}^2
\leq 
2\epsilon'Z(\hat{\gamma}-\gamma^*)+2\lambda_{N,T}\left(\enVert[1]{\beta^*}_{\ell_1}-\enVert[1]{\hat{\beta}}_{\ell_1}\right)+2\mu_{N,T}\left(\enVert[0]{c^*}_{\ell_1}-\enVert[0]{\hat{c}}_{\ell_1}\right)\label{start}
\end{align}
So to bound $\enVert[1]{Z(\hat{\gamma}-\gamma^*)}^2$ it is sufficient to bound $2\epsilon'Z(\hat{\gamma}-\gamma^*)$. Note that on $\mathcal{A}_{N,T}$ one has
\begin{align*}
2\epsilon'Z(\hat{\gamma}-\gamma^*)
&=
2\epsilon'X(\hat{\beta}-\beta^*)+2\epsilon'D(\hat{c}-c^*)\\
&\leq
2\enVert[1]{\epsilon'X}_{\ell_\infty}\enVert[0]{\hat{\beta}-\beta^*}_{\ell_1}+2\enVert[1]{\epsilon'D}_{\ell_\infty}\enVert[0]{\hat{c}-c^*}_{\ell_1}\\
&\leq
\lambda_{N,T}\enVert[1]{\hat{\beta}-\beta^*}_{\ell_1}+\mu_{N,T}\enVert[1]{\hat{c}-c^*}_{\ell_1}
\end{align*}
Putting things together, on $\mathcal{A}_{N,T}$,
\begin{align*}
&\enVert[1]{Z(\hat{\gamma}-\gamma^*)}^2\\
&\leq
\lambda_{N,T}\enVert[1]{\hat{\beta}-\beta^*}_{\ell_1}+2\lambda_{N,T}\left(\enVert[1]{\beta^*}_{\ell_1}-\enVert[1]{\hat{\beta}}_{\ell_1}\right)+\mu_{N,T}\enVert[1]{\hat{c}-c^*}_{\ell_1}+2\mu_{N,T}\left(\enVert[0]{c^*}_{\ell_1}-\enVert[0]{\hat{c}}_{\ell_1}\right)
\end{align*}
Adding $\lambda_{N,T}\enVert[1]{\hat{\beta}-\beta^*}_{\ell_1}$ and $\mu_{N,T}\enVert[1]{\hat{c}-c^*}_{\ell_1}$ yields
\begin{align}
&\enVert[1]{Z(\hat{\gamma}-\gamma^*)}^2+\lambda_{N,T}\enVert[1]{\hat{\gamma}-\gamma^*}_{\ell_1}+\mu_{N,T}\enVert[1]{\hat{c}-c^*}_{\ell_1}\notag\\
&\leq
2\lambda_{N,T}\left(\enVert[1]{\hat{\beta}-\beta^*}_{\ell_1}+\enVert[1]{\beta^*}_{\ell_1}-\enVert[1]{\hat{\beta}}_{\ell_1}\right)+2\mu_{N,T}\left(\enVert[0]{\hat{c}-c^*}_{\ell_1}+\enVert[0]{c^*}_{\ell_1}-\enVert[0]{\hat{c}}_{\ell_1}\right)\label{A1}
\end{align}
Notice that 
\begin{align*}
\enVert[1]{\hat{\beta}-\beta^*}_{\ell_1}+\enVert[1]{\beta^*}_{\ell_1}-\enVert[1]{\hat{\beta}}_{\ell_1}=\enVert[1]{\hat{\beta}_{J_1}-\beta^*_{J_1}}_{\ell_1}+\enVert[1]{\beta^*_{J_1}}_{\ell_1}-\enVert[1]{\hat{\beta}_{J_1}}_{\ell_1}
\end{align*}
In addition, $\enVert[1]{\hat{\beta}_{J_1}-\beta^*_{J_1}}_{\ell_1}+\enVert[1]{\beta^*_{J_1}}_{\ell_1}-\enVert[1]{\hat{\beta}_{J_1}}_{\ell_1}\leq 2\enVert[1]{\hat{\beta}_{J_1}-\beta^*_{J_1}}_{\ell_1}$
by continuity of the norm. By exactly the same arguments $\enVert[0]{\hat{c}-c^*}_{\ell_1}+\enVert[0]{c^*}_{\ell_1}-\enVert[0]{\hat{c}}_{\ell_1} \leq 2\enVert[1]{\hat{c}_{J_2}-c^*_{J_2}}_{\ell_1}$. Using these estimates in (\ref{A1}) yields inequality (\ref{IQ1}).
Next notice that (\ref{IQ1}) gives
\begin{align*}
\lambda_{N,T}\enVert[1]{\hat{\beta}-\beta^*}_{\ell_1}+\mu_{N,T}\enVert[1]{\hat{c}-c^*}_{\ell_1}
\leq
4\lambda_{N,T}\enVert[1]{\hat{\beta}_{J_1}-\beta^*_{J_1}}_{\ell_1}+4\mu_{N,T}\enVert[1]{\hat{c}_{J_2}-c^*_{J_2}}_{\ell_1}
\end{align*}
which is equivalent to
\begin{align*}
\lambda_{N,T}\enVert[1]{\hat{\beta}_{J_1^c}-\beta^*_{J_1^c}}_{\ell_1}+\mu_{N,T}\enVert[1]{\hat{c}_{J_2^c}-c^*_{J_2^c}}_{\ell_1}
\leq
3\lambda_{N,T}\enVert[1]{\hat{\beta}_{J_1}-\beta^*_{J_1}}_{\ell_1}+3\mu_{N,T}\enVert[1]{\hat{c}_{J_2}-c^*_{J_2}}_{\ell_1}
\end{align*}
and establishes inequality (\ref{IQ2}). 
\end{proof}

\begin{proof}[Proof of Theorem \ref{thm1}]
By (\ref{IQ1}) of Lemma \ref{masterlemma} (which is valid on $\mathcal{A}_{N,T}$)
\begin{align}
\enVert[1]{Z(\hat{\gamma}-\gamma^*)}^2
\leq
4\lambda_{N,T}\enVert[1]{\hat{\beta}_{J_1}-\beta^*_{J_1}}_{\ell_1}+4\mu_{N,T}\enVert[1]{\hat{c}_{J_2}-c^*_{J_2}}_{\ell_1}\label{IQcorr}
\end{align}
Next, note that for $b=S^{-1}\delta$ where $b$ is partitioned as $b=({b^1}',{b^2}')'$ with $b^1$ being a $p\times 1$ vector and $b^2$ an $N\times 1$ vector, the restricted eigenvalue condition (\ref{RE}) may be formulated equivalently as 
\begin{align*}
\kappa_{\psi_{N,T}}^2(r_1,r_2)
=
\min\Bigg\{\frac{\enVert[0]{Zb}^2}{\enVert{Sb}^2}: b \in \mathbb{R}^{p+N}\setminus\cbr[0]{0},\ |R_1|\leq r_1,\ |R_2|\leq r_2,\\ \lambda_{N,T}\enVert[1]{b_{R_1^c}^1}_{\ell_1}+\mu_{N,T}\enVert[1]{b_{R_2^c}^2}_{\ell_1}\leq 3\lambda_{N,T}\enVert[1]{b_{R_1}^1}_{\ell_1}+3\mu_{N,T}\enVert[1]{b_{R_2}^2}_{\ell_1}\Bigg\}>0
\end{align*}
Hence, the restricted eigenvalue condition (which is applicable due to (\ref{IQ2}) yields
\begin{align}
\enVert[1]{Z(\hat{\gamma}-\gamma^*)}^2
\geq
\kappa_{\Psi_{N,T}}^2\enVert{S(\hat{\gamma}-\gamma^*)}^2
\geq
\kappa^2/2\left[NT\enVert[1]{\hat{\beta}-\beta^*}^2+T\enVert[1]{\hat{c}-c^*}^2\right]\label{LHS1}
\end{align}
where the last estimate holds on $\mathcal{B}_{N,T}$. By Jensen's inequality
\begin{align}
4\lambda_{N,T}\enVert[1]{\hat{\beta}_{J_1}-\beta^*_{J_1}}_{\ell_1}+4\mu_{N,T}\enVert[1]{\hat{c}_{J_2}-c^*_{J_2}}_{\ell_1}
&\leq
4\lambda_{N,T}\sqrt{s_1}\enVert[1]{\hat{\beta}_{J_1}-\beta^*_{J_1}}+4\mu_{N,T}\sqrt{s_2}\enVert[1]{\hat{c}_{J_2}-c^*_{J_2}}\notag\\
&\leq
4\lambda_{N,T}\sqrt{s_1}\enVert[1]{\hat{\beta}-\beta^*}+4\mu_{N,T}\sqrt{s_2}\enVert[1]{\hat{c}-c^*}\label{RHS1}
\end{align}
Inserting (\ref{LHS1}) and (\ref{RHS1}) into (\ref{IQcorr}) yields
\begin{align*}
\frac{\kappa^2}{2}\left[NT\enVert[1]{\hat{\beta}-\beta^*}^2+T\enVert[1]{\hat{c}-c^*}^2\right]
\leq
4\lambda_{N,T}\sqrt{s_1}\enVert[1]{\hat{\beta}-\beta^*}+4\mu_{N,T}\sqrt{s_2}\enVert[1]{\hat{c}-c^*}
\end{align*}
or equivalently,
\begin{align}
\enVert[1]{\hat{\beta}-\beta^*}^2+\frac{1}{N}\enVert[1]{\hat{c}-c^*}^2-\frac{8\lambda_{N,T}\sqrt{s_1}}{\kappa^2NT}\enVert[1]{\hat{\beta}-\beta^*}-\frac{8\mu_{N,T}\sqrt{s_2}}{\kappa^2NT}\enVert[1]{\hat{c}-c^*}\leq 0 \label{quadineqq}
\end{align}
For $x=\enVert[0]{\hat{\beta}-\beta^*}$ and $y=\enVert[0]{\hat{c}-c^*}$ this can be written as a quadratic inequality in two variables:
\begin{align}
x^2-ax+by^2-cy\leq 0,\ x,y\geq 0\label{quadineq}
\end{align}
\footnote{Note that this inequality is trivially satisfied by $x=y=0$, corresponding to no estimation error. However, we are looking for an upper bound on $x$ and $y$, i.e. the largest possibe values that satisfy (\ref{quadineq}) and hence (\ref{quadineqq}).} with $a=\frac{8\lambda_{N,T}\sqrt{s_1}}{\kappa^2NT},\ b=\frac{1}{N}$ and $c=\frac{8\mu_{N,T}\sqrt{s_2}}{\kappa^2NT}$. First bound $x=\enVert[0]{\hat{\beta}-\beta^*}$. For every $y$ the values of $x$ that satisfy  (\ref{quadineq}) form an interval in $\mathbb{R}_+$. The right end point of this interval is the desired upper bound on $x$. Clearly, by the solution formula for the roots of a second degree polynomial, this right end point is a decreasing function in $by^2-cy$. Hence, we first minimize the polynomial $by^2-cy$ to find the largest possible value of $x$ which satisfies (\ref{quadineq}). This yields $y=\frac{c}{2b}$ and the corresponding value of $by^2-cy$ is $-\frac{c^2}{4b}$. Hence, our desired upper bound on $x$ is the largest solution of $x^2-ax-\frac{c^2}{4b}\leq 0$. By the standard solution formula for the roots of a quadratic polynomial this yields 
\begin{align}
\enVert[1]{\hat{\beta}-\beta^*}=x\leq \frac{a+\sqrt{a^2+c^2/b}}{2}\label{boundx}
\end{align}
Switching the roles of $x$ and $y$, one gets a similar bound on $y=\enVert[0]{\hat{c}-c^*}$, namely
\begin{align}
\enVert[1]{\hat{c}-c^*}=y\leq \frac{c+\sqrt{c^2+ba^2}}{2b}\label{boundy}
\end{align}
Inserting the definitions of $a,b$ and $c$ into (\ref{boundx}) yields
\begin{align*}
\enVert[1]{\hat{\beta}-\beta^*}
\leq
\frac{\frac{8\lambda_{N,T}\sqrt{s_1}}{\kappa^2NT}+\sqrt{\del[2]{\frac{8\lambda_{N,T}\sqrt{s_1}}{\kappa^2NT}}^2+\del[2]{\frac{8\mu_{N,T}\sqrt{s_2}}{\kappa^2NT}}^2N}}{2}
\leq
\frac{8\lambda_{N,T}\sqrt{s_1}}{\kappa^2NT}+\frac{4\mu_{N,T}\sqrt{s_2}}{\kappa^2\sqrt{N}T}
\end{align*}
by subadditivity of $x\mapsto\sqrt{x}$. Similarly,
\begin{align*}
\enVert[1]{\hat{c}-c^*}
\leq
\frac{\frac{8\mu_{N,T}\sqrt{s_2}}{\kappa^2NT}+\sqrt{\del[2]{\frac{8\mu_{N,T}\sqrt{s_2}}{\kappa^2NT}}^2+\frac{1}{N}\del[2]{\frac{8\lambda_{N,T}\sqrt{s_1}}{\kappa^2NT}}^2}}{2/N}
\leq
\frac{8\mu_{N,T}\sqrt{s_2}}{\kappa^2T}+\frac{4\lambda_{N,T}\sqrt{s_1}}{\kappa^2\sqrt{N}T}
\end{align*}
\end{proof}

Before stating the next lemma we shall remark that when no further distinction between subscripts $i$ and $t$ is needed we shall sometimes use $x_{j,k}$ to denote the $j$th entry of the $k$th variable $x_k=(x_{1,1,k},x_{1,2,k},...,x_{1,T,k},x_{2,1,k},...,x_{N,T,k})'$ with $1\leq j \leq NT$. Similarly, we will write $\epsilon_j$ for the $j$th entry of $\epsilon=(\epsilon_{1,1},\epsilon_{1,2},...,\epsilon_{1,T},\epsilon_{2,1},...,\epsilon_{N,T})'$ $1\leq j\leq NT$.

\begin{lemma}\label{AboundLp}
Let $\lambda_{N,T}=4a_{N,T}p^{1/r}(NT)^{1/2}\max_{1\leq t\leq T}\enVert[0]{\epsilon_{1,t}}_{L_r}$ and\\ $\mu_{N,T}=4a_{N,T}N^{1/r}T^{1/2}\max_{1\leq t\leq T}\enVert[0]{\epsilon_{1,t}}_{L_r}$ for some sequence $a_{N,T}$. Then, under assumption A1) and  A2a)
\begin{align}
P\del[1]{\mathcal{A}_{N,T}^c}=P\del[3]{\cbr[3]{\enVert[0]{X'\epsilon}_{\ell_\infty}>\frac{\lambda_{N,T}}{2}}\cup \cbr[3]{\enVert[0]{D'\epsilon}_{\ell_\infty}>\frac{\mu_{N,T}}{2}}}\leq 2\del{\frac{C_r}{a_{N,T}}}^r\label{bothbound}
\end{align}
\end{lemma}

\begin{proof}[Proof of Lemma \ref{AboundLp}]
First bound $\enVert[1]{\max_{1\leq k \leq p} \envert[1]{\sum_{j=1}^{NT}x_{j,k}\epsilon_{j}}}_{L_r}$. To this end, note that for any collection of random variables $\cbr[0]{U_{k}}_{k=1}^p\subseteq L_r$,
\begin{align*}
\enVert{\max_{1\leq k \leq p}U_k}_{L_r}
=
[E(|\max_{1\leq k \leq p}U_k|^r)]^{1/r}
\leq
\left[E\del[4]{\sum_{k=1}^p|U_k|^r}\right]^{1/r} 
\leq
p^{1/r}\max_{1\leq k \leq p}\enVert{U_k}_{L_r}
\end{align*}
Next, bound $\enVert[1]{\sum_{j=1}^{NT}x_{j,k}\epsilon_{j}}_{L_r}$ uniformly in $1\leq k \leq p$. Denote by $\mathcal{F}_n=\sigma\del[1]{\cbr{X, \epsilon_j,\ 1\leq j \leq n}}$ the $\sigma$-field generated by $X$ and $\epsilon_{j},\ 1\leq j \leq n$ and set $S_{n,k}=\sum_{j=1}^nx_{j,k}\epsilon_{j}$. Then $\cbr[0]{(S_{n,k}, \mathcal{F}_n),\ 1\leq n\leq NT}$ is a martingale for all $1\leq k\leq p$ under assumptions A1 and the given moment assumptions\footnote{In particular,
\begin{align*}
E\del[2]{\sum_{j=1}^{n}x_{j,k}\epsilon_j|\mathcal{F}_{n-1}}
&=
\sum_{j=1}^{n-1}x_{j,k}\epsilon_j+x_{n,k}E\del[1]{\epsilon_n|\mathcal{F}_{n-1}}
\stackrel{(1)}{=}
\sum_{j=1}^{n-1}x_{j,k}\epsilon_j+x_{n,k}E\del[1]{\epsilon_n|\sigma\del[1]{\cbr[0]{\epsilon_1,...,\epsilon_{n-1}}}}\\
&\stackrel{(2)}{=}\sum_{j=1}^{n-1}x_{j,k}\epsilon_j+x_{n,k}E\del[1]{\epsilon_{i,t}|\mathcal{F}_{i,t-1}}
=\sum_{j=1}^{n-1}x_{j,k}\epsilon_j.
\end{align*}
where (1) follows from the independence of $X$ and $\epsilon_1,...,\epsilon_n$. (2) follows from identifying $n$ with a pair $(i,t)$, using the independent sampling across $i=1,...,N$ and that $\cbr[1]{\epsilon_{i,t},\mathcal{F}_{i,t}}_{t=1}^T$ form a martingale difference sequence.
}
. Hence, by Rosenthal's inequality for martingales (see \cite{hitczenko90} or \cite{hallh80}) for a constant $C_r$ depending only on $r$,\footnote{By independence of $x_{j,k}$ and $\epsilon_j$ their product is in $L_r$ and Rosenthal's inequality yields a nontrivial upper bound.}
\small{
\begin{align*}
\enVert[4]{\sum_{j=1}^{NT}x_{j,k}\epsilon_{j}}_{L_r}
&\leq
C_r\left[\del[4]{E\del[4]{\sum_{j=1}^{NT}E(x_{j,k}^2\epsilon_{j}^2|\mathcal{F}_{j-1})}^{r/2}}^{1/r}+\del[4]{E\left[\max_{1\leq j \leq NT}|x_{j,k}\epsilon_{j}|^r\right]}^{1/r}\right]\\
&\leq
C_r\left[\del[4]{E\del[4]{\sum_{j=1}^{NT}x_{j,k}^2\enVert[0]{\epsilon_j}_{L_2}^2}^{r/2}}^{1/r}+\del[4]{E\left[\sum_{j=1}^{NT}|x_{j,k}\epsilon_{j}|^r\right]}^{1/r}\right]\\
&\leq
C_r\left[\del[4]{(NT)^{r/2-1}\sum_{j=1}^{NT}E|x_{j,k}|^{r}\enVert[0]{\epsilon_j}_{L_2}^{r}}^{1/r}+(NT)^{1/r}\max_{1\leq t \leq T}\enVert{x_{1,t,k}}_{L_r}\enVert{\epsilon_{1,t}}_{L_r}\right]\\
&\leq
C_r\left[\del[4]{(NT)^{r/2-1}NT\max_{1\leq t \leq T}E|x_{1,t,k}|^{r}\enVert[0]{\epsilon_{1,t}}_{L_2}^{r}}^{1/r}+(NT)^{1/r}\max_{1\leq t \leq T}\enVert{x_{1,t,k}}_{L_r}\enVert{\epsilon_{1,t}}_{L_r}\right]\\
&\leq
C_r\left[(NT)^{1/2}\max_{1\leq t \leq T}\enVert{x_{1,t,k}}_{L_r}\enVert[0]{\epsilon_{1,t}}_{L_2}+(NT)^{1/r}\max_{1\leq t \leq T}\enVert{x_{1,t,k}}_{L_r}\enVert{\epsilon_{1,t}}_{L_r}\right]\\
&\leq
2C_r(NT)^{1/2}\max_{1\leq t \leq T}\enVert{\epsilon_{1,t}}_{L_r}\\
\end{align*}
}
In the above display we have used Loeve's $c_r$-inequality and by \cite{hitczenko90} we know that $C_r\leq 10r$. \cite{hitczenko90} actually shows that the optimal constant $C_r\in O(r/\ln(r))$ as $r\rightarrow\infty$. Hence,
\begin{align*}
\enVert[4]{\max_{1\leq k \leq p} \envert[3]{\sum_{j=1}^{NT}x_{j,k}\epsilon_{j}}}_{L_r}
\leq
\max_{1\leq k \leq p}p^{1/r}\enVert[4]{\sum_{j=1}^{NT}x_{j,k}\epsilon_{j}}_{L_r}
\leq
p^{1/r}2C_r(NT)^{1/2}\max_{1\leq t \leq T}\enVert{\epsilon_{1,t}}_{L_r}
\end{align*}
By Markov's inequality,%
\begin{align*}
P\del[4]{\max_{1\leq k \leq p} \envert[3]{\sum_{j=1}^{NT}x_{j,k}\epsilon_{j}}>\frac{\lambda_{N,T}}{2}}
\leq
\frac{1}{\del[1]{\lambda_{N,T}/(4p^{1/r}C_r(NT)^{1/2}\max_{1\leq t \leq T}\enVert{\epsilon_{1,t}}_{L_r})}^r}
=
\left(\frac{C_r}{a_{N,T}}\right)^r
\end{align*}
In a similar way as above it follows by Rosenthal's inequality%
\begin{align*}
\enVert[4]{\sum_{t=1}^{T}\epsilon_{i,t}}_{L_r}
&\leq
C_r\left[\del{\sum_{t=1}^TE(\epsilon_{i,t}^2)}^{1/2}+\del[3]{E\del[2]{\max_{1\leq t \leq T} |\epsilon_{i,t}|^r}}^{1/r}\right]\\
&\leq
C_r\left[T^{1/2}\max_{1\leq t\leq T}\enVert[0]{\epsilon_{1,t}}_{L_2}+T^{1/r}\max_{1\leq t\leq T}\enVert[0]{\epsilon_{1,t}}_{L_r}\right]\\ 
&\leq
2C_rT^{1/2}\max_{1\leq t\leq T}\enVert[0]{\epsilon_{1,t}}_{L_r}
\end{align*}
This implies that
\begin{align*}
\enVert[4]{\max_{1\leq i\leq N}\sum_{t=1}^{T}\epsilon_{i,t}}_{L_r}
\leq
\max_{1\leq i\leq N}N^{1/r}\enVert[4]{\sum_{t=1}^{T}\epsilon_{i,t}}_{L_r}
\leq
N^{1/r}2C_rT^{1/2}\max_{1\leq t\leq T}\enVert[0]{\epsilon_{1,t}}_{L_r}
\end{align*}
And so, by Markov's inequality,
\begin{align*}
P\left(\max_{1\leq i\leq N}\sum_{t=1}^{T}\epsilon_{i,t}>\frac{\mu_{N,T}}{2}\right)
\leq
\frac{1}{\del[1]{\mu_{N,T}/(4N^{1/r}C_rT^{1/2}\max_{1\leq t\leq T}\enVert[0]{\epsilon_{1,t}}_{L_r})}^r}
=\del{\frac{C_r}{a_{N,T}}}^r
\end{align*}
It follows that 
\begin{align*}
P\del[3]{\cbr[3]{\enVert[0]{X'\epsilon}_{\ell_\infty}>\frac{\lambda_{N,T}}{2}}\cup \cbr[3]{\enVert[0]{D'\epsilon}_{\ell_\infty}>\frac{\mu_{N,T}}{2}}}\leq 2\del{\frac{C_r}{a_{N,T}}}^r
\end{align*}
\end{proof}

\begin{lemma}\label{LV}
Let $\cbr{U_i,\mathcal{F}_i}_{i=1}^n$ be a martingale difference sequence and assume that there exist $\delta,M>0$ such that $E\exp(\delta |U_i|)\leq M$ for all $i=1,...,n$. Then, for $a>0$, there exists positive constants $A$ and $B$ such that for all $x\geq a/\sqrt{n}$ 
\begin{align}
P\del[3]{\envert[2]{\sum_{i=1}^nU_i}>nx}
<
Ae^{-B(x^2n)^{1/3}}\label{lv}
\end{align}
\end{lemma}
\begin{proof}
In the proof of their Theorem 3.2 \cite{lesignev01} show that if $E\exp(|U_i|)\leq M$ for all $i=1,...,n$, then for any $x>0$ and $t\in(0,1)$ one has\footnote{See the last expression in the proof of their Theorem 3.2.}
\begin{align}
&P\del[3]{\envert[2]{\sum_{i=1}^nU_i}>nx}\notag\\
<
&\del[3]{2+\frac{M}{(1-t)^2}\sbr[2]{\frac{1}{4}t^{4/3}(x^{-2}n^{-1})^{1/3}+t^{2/3}(x^{-2}n^{-1})^{2/3}+2x^{-2}n^{-1}}}e^{-(1/2)t^{2/3}(x^2n)^{1/3}}\label{lv1}
\end{align}
But note that $P\del[1]{\envert[0]{\sum_{i=1}^nU_i}>nx}=P\del[1]{\envert[0]{\sum_{i=1}^n(\delta U_i)}>n(\delta x)}$ where $\cbr{\delta U_i}_{i=1}^n$, by assumption, now satisfy the conditions of Theorem 3.2 in \cite{lesignev01} and so replacing $x$ by $\delta x$ in (\ref{lv1}) yields
\small{
\begin{align*}
&P\del[3]{\envert[2]{\sum_{i=1}^nU_i}>nx}\\
<
&\del[3]{2+\frac{M}{(1-t)^2}\sbr[2]{\frac{1}{4}t^{4/3}\delta^{-2/3}(x^{-2}n^{-1})^{1/3}+t^{2/3}\delta^{-4/3}(x^{-2}n^{-1})^{2/3}+2\delta^{-2}x^{-2}n^{-1}}}e^{-(1/2)t^{2/3}\delta^{2/3}(x^2n)^{1/3}}
\end{align*}
}
Restricting $x$ to be greater than $a/\sqrt{n}$, implying that $x^{-2}n^{-1}\leq 1/a^2$, and using that $M,t$ and $\delta$ are constants the conclusion of the lemma follows.
\end{proof}

For the proof of Lemma \ref{Aboundexp} below, we shall use Orlicz norms as defined in \cite{vdVW96}: Let $\psi$ be a non-decreasing convex function with $\psi(0)=0$. Then, the Orlicz norm of a random variable $X$ is given by
\begin{align*}
\enVert{X}_\psi=\inf\left\{C>0:E\psi\left(|X|/C\right)\leq 1\right\}
\end{align*}

where, as usual, $\inf \emptyset =\infty$. We will use Orlicz norms for  $\psi(x)=\psi_p(x)=e^{x^p}-1$ for $p=1,2$.

\begin{lemma}\label{Aboundexp}
Assume that assumptions A1 and A2b  are satisfied.  Then, for $a_{N,T}\geq e$
\begin{align*}
P\del[2]{\enVert[0]{X'\epsilon}_{\ell_\infty}\geq \lambda_{N,T}/2}\leq Ap^{1-B\log(a_{N,T})} \text{ for } \lambda_{N,T}=\sqrt{4NT\log(p)^3\log(a_{N,T})^3}\\
\text{and}\\
P\del[2]{\enVert[0]{D'\epsilon}_{\ell_\infty}\geq \mu_{N,T}/2}\leq AN^{1-B\ln(a_{N,T})} \text{ for } \mu_{N,T}=\sqrt{4T\log(N)^3\log(a_{N,T})^3}
\end{align*}
\end{lemma}

\begin{proof}
First note that for all $1\leq j\leq NT$ and $1\leq k\leq p$ one has for all $t>0$
\begin{align*}
P\del[1]{|x_{j,k}\epsilon_{j}|>t}
\leq
P\del[1]{|x_{j,k}|>\sqrt{t}}+P\del[1]{|\epsilon_{j}|>\sqrt{t}}
\leq
K\exp(-Ct)
\end{align*}
and so it follows from Lemma 2.2.1 in \cite{vdVW96} that $\enVert[0]{x_{j,k}\epsilon_j}_{\psi_1}\leq \frac{1+K}{C}$ and so $E\exp\del[1]{\frac{C}{1+K}|x_{j,k}\epsilon_{j}|}\leq 2$ by the definition of the Orlicz-norm. Hence, $\delta=\frac{C}{1+K}$ works in Lemma \ref{LV} for all $1\leq k\leq p$. Next, denote by $\mathcal{F}_n=\sigma\del[1]{\cbr{X, \epsilon_j,\ 1\leq j \leq n}}$ the $\sigma$-field generated by $X$ and $\epsilon_{j},\ 1\leq j \leq n$ and set $S_{n,k}=\sum_{j=1}^nx_{j,k}\epsilon_{j}$. Then it is clear that $\cbr[0]{(S_{n,k}, \mathcal{F}_n),\ 1\leq n\leq NT}$ is a martingale for all $1\leq k\leq p$ \footnote{See the argument in the proof of Lemma \ref{AboundLp}.}. From a union bound it follows from Lemma \ref{LV} (with $a=1$) that\footnote{Lemma \ref{LV} is applicable since $a_{N,T}$ and $p$ are assumed greater than $e$.}
\begin{align*}
P\del[2]{\enVert[0]{X'\epsilon}_{\ell_\infty}\geq \lambda_{N,T}/2}
=
P\del[2]{\enVert[0]{X'\epsilon}_{\ell_\infty}\geq \frac{\lambda_{N,T}/2}{NT}NT}
\leq
pAe^{-B\del[1]{\frac{\lambda_{NT}^2}{4NT}}^{1/3}}
=Ap^{1-B\log(a_{N,T})}
\end{align*}
Next, by the subgaussianity of $\epsilon_{i,t},\ 1\leq i\leq N,\ 1\leq t\leq T$, it follows from Lemma 2.2.1 in \cite{vdVW96} that $\enVert[0]{\epsilon_{i,t}}_{\psi_2}\leq \del[1]{\frac{1+K/2}{C}}^{1/2}$, and so $\enVert[0]{\epsilon_{i,t}}_{\psi_1}\leq \del[1]{\frac{1+K/2}{C}}^{1/2}\log(2)^{-1/2}$ by the second to last inequality on page 95 in \cite{vdVW96}. Hence, $E\exp\del[1]{\del[1]{\frac{C}{1+K/2}}^{1/2}\log(2)^{1/2}|\epsilon_{i,t}|}\leq 2$.\footnote{We note that this estimate is slightly suboptimal since we are not taking full advantage of the subgaussianity of the $\epsilon_{i,t}$ by merely using it to deduce subexponentiality and then invoking Lemma \cite{lesignev01}. One could use the full strength of the subgaussianity by strengthening $E\exp(|\epsilon|)\leq K$ to $E\exp(\epsilon^2)\leq K$ in Lemma 3.2 of \cite{lesignev01}. Doing so, and adjusting Lemma \ref{LV} accordingly yields that the exponent $1/3$ in (\ref{lv}) can be increased to $1/2$ and hence $\mu_{N,T}$ can in turns be reduced to $\sqrt{4T\log(N)^2\log(a_{N,T})^2}
$. As a third route, one could use Hoeffding's inequality in combination with a truncation of the $\epsilon_{i,t}$. This does not reduce $\mu_{N,T}$ significantly either.} Furthermore, for all $i=1,...,N$, $\cbr[0]{\epsilon_{i,t},\mathcal{F}_{i,t}}_{t=1}^T$ form a martingale difference sequence and so by the union bound and Lemma \ref{LV} \footnote{In principle, the constants $A$ and $B$ need not be the same as above but by simply using the worst ones they can be chosen to be identical. Also, we have used $a_{N,T}, N\geq e$.} (with $a=1$)
\begin{align*}
P\del[2]{\enVert[0]{D'e}_{\ell_\infty}\geq \mu_{N,T}/2}
\leq
NP\del[2]{\enVert[0]{D'e}_{\ell_\infty}\geq \frac{\mu_{N,T}/2}{T}T}
\leq
NAe^{-B\del[1]{\frac{\mu_{N,T}^2}{4T}}^{1/3}}
\leq
AN^{1-B\log(a_{N,T})}
\end{align*}
\end{proof}

\begin{lemma}\label{vdGB}
Let $A$ and $B$ be two positive semi-definite $(p+N)\times (p+N)$ matrices and assume that $A$ satisfies the restricted eigenvalue condition RE($s_1,s_2$) for some $\kappa_A$. Then, for $\delta=\max_{1\leq i,j\leq p+N}\envert[0]{A_{i,j}-B_{i,j}}$, one also has $\kappa_B^2\geq \kappa_A^2-16\delta (s_1+s_2)m_{N,T}^2$ where $m_{N,T}=\frac{\lambda_{N,T}}{\sqrt{N}\mu_{N,T}}\vee \frac{\sqrt{N}\mu_{N,T}}{\lambda_{N,T}}$
\end{lemma}

\begin{proof}
Let $x_1$ be $p\times 1$, $x_2$ be $N\times 1$ and define $x=(x_1',x_2')'$ and assume that $\frac{\lambda_{N,T}}{\sqrt{NT}}\enVert[0]{x_{1J_1^c}}_{\ell_1}+\frac{\mu_{N,T}}{\sqrt{T}}\enVert[0]{x_{2J_2^c}}_{\ell_1}\leq 3\frac{\lambda_{N,T}}{\sqrt{NT}}\enVert[0]{x_{1J_1}}_{\ell_1}+3\frac{\mu_{N,T}}{\sqrt{T}}\enVert[0]{x_{2J_2}}_{\ell_1}$. Defining 
\begin{align*}
V=
\del{\begin{array}{cc}
\frac{\lambda_{N,T}}{\sqrt{NT}}I_{|J_1|} & 0\\
0 & \frac{\mu_{N,T}}{\sqrt{T}}I_{|J_2|}\\
\end{array}} \text{and } 
V_c=
\del{\begin{array}{cc}
\frac{\lambda_{N,T}}{\sqrt{NT}}I_{|J_1^c|} & 0\\
0 & \frac{\mu_{N,T}}{\sqrt{T}}I_{|J_2^c|}\\
\end{array}} 
\end{align*}
this can also be expressed as $\enVert{V_cx_{J^c}}_{\ell_1}\leq 3\enVert{Vx_J}_{\ell_1}$. For any (non-zero) $(p+N)\times 1$ vector $x$ satisfying this restriction one has
\begin{align*}
\enVert[1]{x_{J^c}}_{\ell_1}
=
\enVert[1]{V_c^{-1}V_cx_{J^c}}_{\ell_1}
\leq
\enVert[1]{V_c^{-1}}_{\ell_1}\enVert[1]{V_cx_{J^c}}_{\ell_1}
\leq
3\enVert[1]{V_c^{-1}}_{\ell_1}\enVert[1]{Vx_J}_{\ell_1}
\leq
3\enVert[1]{V_c^{-1}}_{\ell_1}\enVert[1]{V}_{\ell_1}\enVert{x_J}_{\ell_1}
\end{align*}
Since
\begin{align*}
\enVert[1]{V_c^{-1}}_{\ell_1}\enVert[1]{V}_{\ell_1}=\frac{\lambda_{N,T}}{\sqrt{N}\mu_{N,T}}\vee \frac{\sqrt{N}\mu_{N,T}}{\lambda_{N,T}}=:m_{N,T}
\end{align*}
one gets
\begin{align*}
\envert[0]{x'Ax-x'Bx}
&=
\envert[0]{x'(A-B)x}
\leq 
\enVert[0]{x}_{\ell_1}\enVert[0]{(A-B)x}_{\ell_\infty}
\leq
\delta\enVert[0]{x}_{\ell_1}^2
\leq 
\delta(\enVert[0]{x_J}_{\ell_1}+\enVert[0]{x_{J^c}}_{\ell_1})^2\\
&\leq
\delta \del[0]{1+3m_{N,T}}^2\enVert{x_J}_{\ell_1}^2
\leq
16\delta (s_1+s_2)m_{N,T}^2\enVert[0]{x_J}^2
\leq
16\delta (s_1+s_2)m_{N,T}^2\enVert[0]{x}^2
\end{align*} 
where the last estimate follows from the fact that $m_{N,T}\geq 1$ and Jensen's inequality. Hence, 
\begin{align*}
x'Bx
\geq
x'Ax - 16\delta (s_1+s_2)m_{N,T}^2\enVert[0]{x}^2
\end{align*}
or equivalently,
\begin{align*}
\frac{x'Bx}{x'x}
\geq
\frac{x'Ax}{x'x}-16\delta (s_1+s_2)m_{N,T}^2
\geq
\kappa_A^2-16\delta (s_1+s_2)m_{N,T}^2
\end{align*}
Minimizing the left hand side over $\cbr[0]{x\in \mathbb{R}^{p+N}\setminus \{0\}: \enVert[0]{V_cx_{J^c}}_{\ell_1}\leq 3\enVert[0]{Vx_J}_{\ell_1}}$ yields the claim.
\end{proof}
In the following two lemmas we shall use Lemma \ref{vdGB} with 
\begin{align*}
A=
\Gamma
=
\del{\begin{array}{cc}
E\del[1]{\frac{X'X}{NT}} & 0\\
0 & I_{N}\\
\end{array}} \text{ and }
B=
\Psi_{N,T}
=
\del{\begin{array}{cc}
\frac{X'X}{NT} & \frac{X'D}{\sqrt{N}T}\\
\frac{D'X}{\sqrt{N}T} & I_{N}\\
\end{array}}
\end{align*}
in order to establish that $\Psi_{N,T}$ satisfies the restricted eigenvalue condition with high probability. Furthermore, define
\begin{align}
\mathcal{\tilde{B}}_{N,T}=\cbr[2]{\max_{1\leq i,j\leq p+N}{|\Psi_{T,i,j}-\Gamma_{i,j}|}\leq \frac{\kappa^2}{32(s_1+s_2)m_{N,T}^2}}\label{tildeB}
\end{align}

\begin{lemma}\label{BboundLp}
Set $\lambda_{N,T}=4a_{N,T}p^{1/r}(NT)^{1/2}\max_{1\leq t\leq T}\enVert[0]{\epsilon_{1,t}}_{L_r}$ and \\$\mu_{N,T}=4a_{N,T}N^{1/r}T^{1/2}\max_{1\leq t\leq T}\enVert[0]{\epsilon_{1,t}}_{L_r}$. Under assumptions A1 and A2a, $P\del[1]{\kappa_{\Psi_T}^2\geq \kappa^2/2}\geq P(\mathcal{\tilde{B}}_{N,T})\geq 1-D_r\frac{(p^2+Np)(s_1+s_2)^{r/2}\del[0]{\frac{p}{N}\vee \frac{N}{p}}}{\kappa^rN^{r/4}}$ for a constant $D_r$ only depending on $r$.
\end{lemma}
\begin{proof}
By Lemma \ref{vdGB} it follows that ${\kappa_{\Psi_{N,T}}\geq \kappa^2/2}$ on $\tilde{\mathcal{B}}_{N,T}$ Since the lower right $N\times N$ blocks of $\Psi_{N,T}$ and $\Gamma$ are identical it suffices to bound the entries of $\frac{X'X}{NT}-E\del[1]{\frac{X'X}{NT}}$ and $\frac{X'D}{\sqrt{N}T}$. A typical element of $\frac{X'X}{NT}-E\del[1]{\frac{X'X}{NT}}$ is of the form $\frac{1}{N}\sum_{i=1}^N\del[1]{\frac{1}{T}\sum_{t=1}^T\sbr{x_{i,t,k}x_{i,t,l}-E(x_{i,t,k}x_{i,t,l})}}$ for some $k,l\in \cbr[0]{1,...,p}$. Next note that for any sequence of mean zero i.i.d. variables $Z_1,...,Z_N$ in $L_r$ it follows from Rosenthal's inequality that
\begin{align}
\enVert[4]{\sum_{i=1}^NZ_i}_{L_r}
&\leq
C_r\del[4]{\sbr[3]{\sum_{i=1}^NEZ_i^2}^{1/2}+\sbr[3]{E\max_{1\leq i\leq N}|Z_i|^r}^{1/r}}
\leq
C_r\del[4]{N^{1/2}\enVert{Z_1}_{L_2}+N^{1/r}\enVert{Z_1}_{L_r}}\notag\\
&\leq
2C_rN^{1/2}\enVert{Z_1}_{L_r}\label{Lraux}
\end{align}
Furthermore,
\begin{align*}
\enVert[3]{\frac{1}{T}\sum_{t=1}^T\sbr{x_{1,t,k}x_{1,t,l}-E(x_{1,t,k}x_{1,t,l})}}_{L_{r/2}}
&\leq
\max_{1\leq t\leq T}\enVert{x_{1,t,k}x_{1,t,l}-E(x_{1,t,k}x_{1,t,l})}_{L_{r/2}}\\
&\leq
2\max_{1\leq t\leq T}\enVert{x_{1,t,k}x_{1,t,l}}_{L_{r/2}}
\leq
2
\end{align*}    
where the last estimate follows from the Cauchy-Schwarz inequality. Using this in (\ref{Lraux}) (with $r$ replaced by $r/2$) yields
\begin{align*}
\enVert[3]{\frac{1}{N}\sum_{i=1}^N\del[2]{\frac{1}{T}\sum_{t=1}^T\sbr{x_{i,t,k}x_{i,t,l}-E(x_{i,t,k}x_{i,t,l})}}}_{L_{r/2}}
\leq 
4C_{r/2}N^{-1/2}.
\end{align*} 
Markov's inequality yields that for any $\epsilon>0$
\begin{align}
P\del[3]{\envert[3]{\frac{1}{N}\sum_{i=1}^N\del[2]{\frac{1}{T}\sum_{t=1}^T\sbr{x_{i,t,k}x_{i,t,l}-E(x_{i,t,k}x_{i,t,l})}}}>\epsilon}
\leq
\frac{(4C_{r/2})^{r/2}}{\epsilon^{r/2}N^{r/4}}\label{B1}
\end{align}
Next, consider a typical term in $\frac{X'D}{\sqrt{N}T}$. Such a term is on the form $\frac{\sum_{t=1}^Tx_{i,t,k}}{\sqrt{N}T}$ for $i=1,...,N$ and $k=1,...,p$. Since
\begin{align*}
\enVert[3]{\frac{1}{\sqrt{N}T}\sum_{t=1}^Tx_{i,t,k}}_{L_r}
\leq
\frac{1}{\sqrt{N}}\max_{1\leq t\leq T}\enVert{x_{i,t,k}}_{L_r}
\leq
\frac{1}{\sqrt{N}}
\end{align*}
it follows by Markov's inequality that for any $\epsilon>0$
\begin{align}
P\del[2]{\envert[1]{\frac{1}{\sqrt{N}T}\sum_{t=1}^Tx_{i,t,k}}>\epsilon}
\leq
\frac{1}{\epsilon^{r}N^{r/2}}
=
\frac{1}{(\epsilon^{r/2}N^{r/4})^2}\label{B2}
\end{align}
Combining (\ref{B1}) and (\ref{B2}) yields via a union bound over $(p^2+Np)$ terms
\begin{align*}
P\del[1]{\max_{1\leq i,j\leq p+N}|A_{i,j}-B_{i,j}|>\epsilon}
\leq
(p^2+Np)\del[3]{\frac{(4C_{r/2})^{r/2}}{\epsilon^{r/2}N^{r/4}}\vee \frac{1}{(\epsilon^{r/2}N^{r/4})^2}}
\leq
D_r\frac{p^2+Np}{\epsilon^{r/2}N^{r/4}}
\end{align*}
\footnote{Note that the first estimate in the display may be replaced by the slightly sharper estimate
\begin{align*}
P\del[1]{\max_{1\leq i,j\leq p+N}|A_{i,j}-B_{i,j}|>\epsilon}
\leq
p^2\frac{(4C_{r/2})^{r/2}}{\epsilon^{r/2}N^{r/4}}+ Np\frac{1}{(\epsilon^{r/2}N^{r/4})^2}
\end{align*}
However, for $p\geq N$ this will lead to no improvement asymptotically, while the improvement is minor for $N>p$.
}where the last estimate follows from the fact that without loss of generality (since otherwise the upper bound is greater than one) one may assume $\epsilon^{r/2}N^{r/4}\geq 1$ and so $\epsilon^{r/2}N^{r/4}\leq (\epsilon^{r/2}N^{r/4})^2$. $D_r=([4C_{r/2}]^{r/2}\vee 1)$ is a constant only depending on $r$. Using $\epsilon=\frac{\kappa^2}{32(s_1+s_2)m_{N,T}^2}$ yields the lemma upon noting that $m_{N,T}=\del[1]{\frac{p}{N}\vee \frac{N}{p}}^{1/r}$ and merging all constants into $D_r$.
\end{proof}

\begin{lemma}\label{Bboundexp}
Set $\lambda_{N,T}=\sqrt{4NT\log(p)^3\log(a_{N,T})^3}$ and $\mu_{N,T}=\sqrt{4T\log(N)^3\log(a_{N,T})^3}$. Furthermore, let $t=\frac{\kappa^2}{(s_1+s_2)\del[1]{\frac{
\ln(p)}{\ln(N)}\vee \frac{\ln(N)}{\ln(p)}}^3}$ and let $Nt^2\geq 1$. Then, under assumptions A1 and A2b), $P\del[1]{\kappa_{\Psi_T}^2\geq \kappa^2/2}\geq P(\mathcal{\tilde{B}}_{N,T})\geq 1-
A(p^2+Np) e^{-B(t^2N)^{1/3}}$ for absolute constants $A$ and $B$.
\end{lemma}

\begin{proof}
By Lemma \ref{vdGB} it follows that ${\kappa_{\Psi_{N,T}}^2\geq \kappa^2/2}$ on $\tilde{\mathcal{B}}_{N,T}$. Since the lower right $N\times N$ blocks of $\Psi_{N,T}$ and $\Gamma$ are identical it suffices to bound the entries of $\frac{X'X}{NT}-E\del[1]{\frac{X'X}{NT}}$ and $\frac{X'D}{\sqrt{N}T}$. A typical element of $\frac{X'X}{NT}-E\del[1]{\frac{X'X}{NT}}$ is of the form $\frac{1}{N}\sum_{i=1}^N\del[1]{\frac{1}{T}\sum_{t=1}^T\sbr{x_{i,t,k}x_{i,t,l}-E(x_{i,t,k}x_{i,t,l})}}$ for some $k,l\in \cbr[0]{1,...,p}$. 
First, note that for all $1\leq i\leq N$, $1\leq t\leq T$ and $1\leq k,l\leq p$ one has for all $\epsilon>0$
\begin{align*}
P\del[1]{|x_{i,t,k}x_{i,t,l}|>\epsilon}
\leq
P\del[1]{|x_{i,t,k}>\sqrt{\epsilon}}+P\del[1]{|x_{i,t,l}|>\sqrt{\epsilon}}
\leq
K\exp(-C\epsilon)
\end{align*}
and so it follows from Lemma 2.2.1 in \cite{vdVW96} that $\enVert[0]{x_{i,t,k}x_{i,t,l}}_{\psi_1}\leq \frac{1+K}{C}$. Next, note that by subadditivity of the Orlicz norm and Jensen's inequality
\begin{align*}
\enVert[3]{\frac{1}{T}\sum_{t=1}^T\sbr{x_{i,t,k}x_{i,t,l}-E(x_{i,t,k}x_{i,t,l})}}_{\psi_1}
\leq
2\max_{1\leq t\leq T}\enVert{x_{i,t,k}x_{i,t,l}}_{\psi_1} 
\leq
2\frac{1+K}{C}
\end{align*}
Hence, $E\exp(\frac{C}{2(1+K)}|x_{i,t,k}x_{i,t,l}|)\leq 2$. It now follows by the independence across $i=1,...,N$ (using Lemma \ref{LV}) there exists constants $A$ and $B$ such that for any $\epsilon>\frac{1}{32\sqrt{N}}$
\begin{align}
P\del[3]{\envert[2]{\frac{1}{N}\sum_{i=1}^N\del[1]{\frac{1}{T}\sum_{t=1}^T\sbr{x_{i,t,k}x_{i,t,l}-E(x_{i,t,k}x_{i,t,l})}}}\geq\epsilon}
\leq
Ae^{-B(\epsilon^2 N)^{1/3}}\label{C1}
\end{align} 
Next, consider a typical term in $\frac{X'D}{\sqrt{N}T}$. Such a term is on the form $\frac{\sum_{t=1}^Tx_{i,t,k}}{\sqrt{N}T}$ for $i=1,...,N$ and $k=1,...,p$. Since $\enVert[0]{x_{i,t,k}}_{\psi_2}\leq \del[1]{\frac{1+K/2}{C}}^{1/2}$ by Lemma 2.2.1 in \cite{vdVW96} one gets
\begin{align*}
\enVert[3]{\frac{1}{\sqrt{N}T}\sum_{t=1}^Tx_{i,t,k}}_{\psi_2}
\leq
\frac{1}{\sqrt{N}}\max_{1\leq t\leq T}\enVert{x_{i,t,k}}_{\psi_2}
\leq
\frac{1}{\sqrt{N}}\del[2]{\frac{1+K/2}{C}}^{1/2}:=\frac{M}{\sqrt{N}}.
\end{align*}
It follows by Markov's inequality and $1\wedge \psi_2(x)^{-1}=1\wedge (e^{x^2}-1)^{-1}\leq 2e^{-x^2}$ that for any $\epsilon>0$
\begin{align}
P\del[2]{\envert[1]{\frac{1}{\sqrt{N}T}\sum_{t=1}^Tx_{i,t,k}}>\epsilon}
\leq
1\wedge \frac{1}{e^{(\epsilon \sqrt{N}/M)^2}-1}
\leq
2e^{-(\epsilon \sqrt{N}/M)^2}
\leq
Ae^{-B\epsilon^2 N}\label{C2}
\end{align}
where the last estimate follows by choosing $A$ and $B$ sufficiently large/small for (\ref{C1}) and (\ref{C2}) both to be valid.
Combining (\ref{C1}) and (\ref{C2}) yields via a union bound over $(p^2+Np)$ terms
\begin{align*}
P\del[1]{\max_{1\leq i,j\leq p+N}|A_{i,j}-B_{i,j}|>\epsilon}
\leq
A(p^2+Np)\del[2]{e^{-B(\epsilon^2 N)^{1/3}}\vee e^{-B\epsilon^2 N}}
\end{align*}
Using $\epsilon=\frac{\kappa^2}{32(s_1+s_2)m_{N,T}^2}$ with $m_{N,T}=\frac{
\ln(p)^{3/2}}{\ln(N)^{3/2}}\vee \frac{\ln(N)^{3/2}}{\ln(p)^{3/2}}$ means that $\epsilon \geq \frac{1}{32\sqrt{N}}$ since $t^2N\geq 1$. Hence, 
\begin{align*}
P\del[1]{\max_{1\leq i,j\leq p+N}|A_{i,j}-B_{i,j}|>\epsilon}
&\leq
A(p^2+Np)\del[2]{e^{-B((1/32)^2t^2 N)^{1/3}}\vee e^{-B(1/32)^2t^2 N}}\\
&\leq
A\del[0]{p^2+Np}e^{-B(t^2 N)^{1/3}}
\end{align*}
where the $(1/32)^2$ have been merged into $B$ and we have used that $t^2N\geq 1$.
\end{proof}

\begin{proof}[Proof of Theorem \ref{thm2}]
$P\del[0]{\mathcal{A}_{N,T}\cap \mathcal{B}_{N,T}}\geq 1-2\del{\frac{C_r}{a_{N,T}}}^r-D_r\frac{(p^2+Np)(s_1+s_2)^{r/2}\del[0]{\frac{p}{N}\vee \frac{N}{p}}}{\kappa^rN^{r/4}}$ follows from Lemmas \ref{AboundLp} and \ref{BboundLp}. Hence, the estimates in Theorem \ref{thm1} are valid with at least this probability. Inserting the definitions of $\lambda_{N,T}$ and $\mu_{N,T}$ into (\ref{IQthm1}) and (\ref{IQ2thm1}) yields (\ref{IQ1thm2}) and (\ref{IQ2thm2}).
\end{proof}

\begin{proof}[Proof of Theorem \ref{thm3}]
The lower bound on $P\del[0]{\mathcal{A}_{N,T}\cap \mathcal{B}_{N,T}}$ follows by combining Lemmas \ref{Aboundexp} and \ref{Bboundexp}. Hence, the estimates in Theorem \ref{thm1} are valid with a probability bounded from below by this estimate. Inserting the definitions of $\lambda_{N,T}$ and $\mu_{N,T}$ into (\ref{IQthm1}) and (\ref{IQ2thm1}) yields (\ref{IQ1thm3}) and (\ref{IQ2thm3}).
\end{proof}

Before we prove Theorem \ref{lowerbound} below we define the weighted Lasso as the minimizer of the following objective function,
\begin{align}
\enVert{y-Z\gamma}^2+2\sum_{j=1}^{p+N}w_{j}|\gamma_j|\label{aLassoobj2}
\end{align}
where $w_j,\ j=1,...,p+N$ are the weights. Note that in the plain panel Lasso, $w_j=\lambda_{N,T}$ for all $j=1,...,p$ and $w_j=\mu_{N,T}$ for $j=p+1,...,p+N$.
From standard convex analysis we know that a vector $\tilde{\gamma}$ minimizes (\ref{aLassoobj2}) if and only if there exists a subgradient $v$ of $\enVert[0]{\gamma}_{\ell_1}$ such that
\begin{align}
-Z_j'(y-Z\tilde{\gamma})+w_jv_j=0 \text{ for all } j=1,...,p+N \label{FOC}
\end{align}
where $v_j=\sgn(\tilde{\gamma}_j)$ if $\tilde{\gamma}_j\neq 0$ and $v_j\in \sbr{-1,1}$ if $\tilde{\gamma}_j=0$.
The following Lemma will be used in the proof of Theorems \ref{lowerbound} and \ref{aLassoThm1}.
\begin{lemma}\label{wLassoLemma}
Suppose that $|v_j|<1$ for all $\tilde{\gamma}_j=0$ in (\ref{FOC}) and that $Z_J'Z_J$ is invertible. Then $\sgn(\tilde{\gamma})=\sgn(\gamma^*)$ if
\begin{align}
\sgn\del[2]{\gamma^*_J+\del{Z_J'Z_J}^{-1}\sbr{Z_J'\epsilon -r_J}}=\sgn(\gamma^*_J)\label{Cond1}
\end{align}
(here $r$ is the $(p+N)\times 1$ vector with $j$th entry $w_jv_j$) and
\begin{align}
\envert[1]{-Z_{j}'Z_J\del{Z_J'Z_J}^{-1}\sbr{Z_J'\epsilon -r_J}+Z_{j}'\epsilon}< w_j\label{Cond2}
\end{align}
for all $j\in J^c$
\end{lemma}

\begin{proof}
The proof combines ideas from \cite{wainwright09} and \cite{zhouvdGB09}. Clearly, $\sgn(\tilde{\gamma})=\sgn(\gamma^*)$ if and only if i) $\tilde{\gamma}$ solves (\ref{FOC}) and ii) $\sgn(\tilde{\gamma})=\sgn(\gamma^*)$. Using $y=Z\gamma^*+\epsilon$ the first order condition (\ref{FOC}) is equivalent to
\begin{align*}
Z'Z(\tilde{\gamma}-\gamma^*)-Z'\epsilon+r=0
\end{align*} 
Using $\tilde{\gamma}_{J^c}=\gamma^*_{J^c}=0$ it follows by the invertibility of $Z_J'Z_J$ that
\begin{align}
\tilde{\gamma}_J-\gamma^*_J=\del{Z_J'Z_J}^{-1}\sbr{Z_J'\epsilon -r_J}\label{gammarel}
\end{align}
which yields $\sgn(\tilde{\gamma}_J)=\sgn(\gamma^*_J)$ under the stated conditions. Furthermore, we have
\begin{align*}
0=
Z_{J^c}'Z_J(\tilde{\gamma}_J-\gamma^*_J)-Z_{J^c}'\epsilon+r_{J^c}
=
Z_{J^c}'Z_J\del{Z_J'Z_J}^{-1}\sbr{Z_J'\epsilon -r_J}-Z_{J^c}'\epsilon+r_{J^c}
\end{align*}
Hence, we must have
\begin{align*}
w_jv_j=r_j=-Z_{j}'Z_J\del{Z_J'Z_J}^{-1}\sbr{Z_J'\epsilon -r_J}+Z_{j}'\epsilon
\end{align*}
for all $j\in J^c$ which means (using $|v_j|<1$)
\begin{align}
\envert[1]{-Z_{j}'Z_J\del{Z_J'Z_J}^{-1}\sbr{Z_J'\epsilon -r_J}+Z_{j}'\epsilon}< w_j
\end{align}
for all $j\in J^c$. Next, $|v_j|<1$ may be used to show that \textit{any} solution $\bar{\gamma}$ of the minimization problem must have $\bar{\gamma}_j=0$ if $\tilde{\gamma}_j$=0. This can be done by mimicking the argument in the proof of Lemma 2.1 in  \cite{buhlmannvdg11}. Finally, using that $\tilde{\gamma}_{J^c}=0$ and that $Z_J'Z_J$ is invertible (\ref{aLassoobj2}) is seen to be strictly convex and so $\tilde{\gamma}'=\del[1]{{\gamma^*}'+\del{Z_J'Z_J}^{-1}\sbr{Z_J'\epsilon -r_J}',0'}$ is indeed the only solution.
\end{proof}

\begin{proof}[Proof of Theorem \ref{lowerbound}]
By (\ref{gammarel}) one gets
\begin{align*}
S_{J,J}(\tilde{\gamma}_J-\gamma^*_J)=\del{S_{J,J}^{-1}Z_J'Z_JS_{J,J}^{-1}}^{-1}\sbr{S_{J,J}^{-1}Z_J'\epsilon -S_{J,J}^{-1}r_J}
\end{align*}
where $r_J$ is the $J\times 1$ vector with the first $|J_1|$ entries equaling $\sgn(\beta^*_{J_1})\lambda_{N,T}$ and the last $|J_2|$ entries equaling  $\sgn(c^*_{J_2})\mu_{N,T}$. This implies
\begin{align}
\enVert{S_{J,J}(\tilde{\gamma}_J-\gamma^*_J)}\geq \enVert[2]{\del{S_{J,J}^{-1}Z_J'Z_JS_{J,J}^{-1}}^{-1}S_{J,J}^{-1}r_J}-\enVert[2]{\del{S_{J,J}^{-1}Z_J'Z_JS_{J,J}^{-1}}^{-1}S_{J,J}^{-1}Z_J'\epsilon}\label{RT}
\end{align}
Next, note that using arguments similar to those in Lemma \ref{vdGB} it is seen that on $\tilde{\mathcal{B}}_{N,T}$ defined in (\ref{tildeB}) one has $\phi_{\max}\del[1]{S_{J,J}^{-1}Z_J'Z_JS_{J,J}^{-1}}\leq 2\phi_{\max}(\Gamma_{J,J})$ and so
\begin{align*}
\enVert[2]{\del{S_{J,J}^{-1}Z_J'Z_JS_{J,J}^{-1}}^{-1}S_{J,J}^{-1}r_J}^2
\geq 
\phi_{\min}^2\sbr[2]{\del[1]{S_{J,J}^{-1}Z_J'Z_JS_{J,J}^{-1}}^{-1}}\del[2]{\frac{|J_1|\lambda_{N,T}^2}{NT}+\frac{|J_2|\mu_{N,T}^2}{T}}\\
\geq
\frac{1}{\phi_{\max}^2\del[1]{S_{J,J}^{-1}Z_J'Z_JS_{J,J}^{-1}}}\del[2]{\frac{|J_1|\lambda_{N,T}^2}{NT}+\frac{|J_2|\mu_{N,T}^2}{T}}
\geq
\frac{1}{2\phi_{\max}^2\del[1]{\Gamma_{J,J}}}\del[2]{\frac{|J_1|\lambda_{N,T}^2}{NT}+\frac{|J_2|\mu_{N,T}^2}{T}}
\end{align*}
Turning to the second term in (\ref{RT}), by the independence of $Z_J$ and $\epsilon$ and the gaussianity of $\epsilon$, it follows that conditional on $Z_J$, $\del{S_{J,J}^{-1}Z_J'Z_JS_{J,J}^{-1}}^{-1}S_{J,J}^{-1}Z_J'\epsilon$ is 
gaussian with mean zero and covariance $\sigma^2\del[2]{S_{J,J}^{-1}Z_J'Z_JS_{J,J}^{-1}}^{-1}$. Hence, for any $s>0$ letting $\tilde{\epsilon}\in \mathbb{R}^{|J|}$ be a standard gaussian vector we have
\begin{align*}
P\del[2]{\enVert[2]{\del{S_{J,J}^{-1}Z_J'Z_JS_{J,J}^{-1}}^{-1}S_{J,J}^{-1}Z_J'\epsilon}^2\leq s}
=
P\del[2]{\tilde{\epsilon}'\sigma^2\del[2]{S_{J,J}^{-1}Z_J'Z_JS_{J,J}^{-1}}^{-1}\tilde{\epsilon}\leq s}
\end{align*}
Furthermore, 
\begin{align*}
\tilde{\epsilon}'\sigma^2\del[2]{S_{J,J}^{-1}Z_J'Z_JS_{J,J}^{-1}}^{-1}\tilde{\epsilon}
\leq
\sigma^2\tilde{\epsilon}'\tilde{\epsilon}\phi_{\max}\del[2]{\del[2]{S_{J,J}^{-1}Z_J'Z_JS_{J,J}^{-1}}^{-1}}
=
\frac{\sigma^2\tilde{\epsilon}'\tilde{\epsilon}}{\phi_{\min}\del[2]{S_{J,J}^{-1}Z_J'Z_JS_{J,J}^{-1}}}
\end{align*}
But since $\tilde{\epsilon}'\tilde{\epsilon}$ is $\chi^2(|J|)$ it follows from the $\chi^2$-concentration inequality in expression (54a) in \cite{wainwright09}\footnote{More precisely, (54a) in \cite{wainwright09} states that given a centered $\chi^2$-variable $X$ with $d$ degrees of freedom, then for any $t\in (0,1/2)$ one has $P(X\geq d(1+t))\leq \exp\del[1]{-\frac{3}{16}dt^2}$. Hence, for an uncentered $\chi^2$-variable $Y$ with $d$ degrees of freedom 
\begin{align*}
P(Y\geq 3d)\leq P(Y\geq d+(1+t)d)=P(Y-d\geq (1+t)d)=P(X\geq (1+t)d)\leq \exp\del[2]{-\frac{3}{16}dt^2}=\exp\del[2]{-c_1d}
\end{align*}
where the first estimate follows from $d\in (0,1/2)$ and the last equality by fixing some $t\in (0,1/2)$.
} that there exists a constant $c_1$ such that $P(\tilde{\epsilon}'\tilde{\epsilon}\geq 3|J|)\leq \exp\del[1]{-c_1|J|}$
Also, by arguments similar to the one in the proof of Lemma \ref{vdGB} one has on $\tilde{\mathcal{B}}_{N,T}$ that $1/\phi_{\min}\del[2]{S_{J,J}^{-1}Z_J'Z_JS_{J,J}^{-1}}\leq 2/\phi_{\min}(\Gamma_{J,J})$ and so 
\begin{align*}
P\del[2]{\tilde{\epsilon}'\sigma^2\del[2]{S_{J,J}^{-1}Z_J'Z_JS_{J,J}^{-1}}^{-1}\tilde{\epsilon}\leq s}
\geq
1-\exp\del[1]{-c_1|J|}-A(p^2+Np)e^{-B(t^2N)^{1/3}}
\end{align*}
for $s=3|J|\sigma^22/\phi_{\min}(\Gamma_{J,J})$. Hence, with probability at least $1-\exp\del[1]{-c_1|J|}-A(p^2+Np)e^{-B(t^2N)^{1/3}}
$ for constants $d_1, d_2$ and $c$
\begin{align*}
\enVert{S_{J,J}(\tilde{\gamma}_J-\gamma^*_J)}
&\geq
\sqrt{\frac{1}{2\phi_{\max}^2\del[1]{\Gamma_{J,J}}}\del[2]{\frac{|J_1|\lambda_{N,T}^2}{NT}+\frac{|J_2|\mu_{N,T}^2}{T}}}-\sqrt{3\sigma^2|J|\cdot 2/\phi_{\min}(\Gamma_{J,J})}\\
&\geq 
d_1(\sqrt{|J_1|}\lambda_{N,T}/\sqrt{NT}+\sqrt{|J_2|}\mu_{N,T}/\sqrt{T})-d_2(\sqrt{|J_1|}+\sqrt{|J_2|})\\
&=
d_1\sqrt{|J_1|}\lambda_{N,T}/\sqrt{NT}\del[2]{1-\frac{d_2\sqrt{NT}}{d_1\lambda_{N,T}}}+d_1\sqrt{|J_2|}\mu_{N,T}/\sqrt{T}\del[2]{1-\frac{d_2\sqrt{T}}{d_1\mu_{N,T}}}\\
&\geq
c\sqrt{|J_1|}\lambda_{N,T}/\sqrt{NT}+c\sqrt{|J_2|}\mu_{N,T}/\sqrt{T}\\
&=
c_2\xi_{N,T}
\end{align*}
where the second estimate used Jensen's inequality on the concave map $x\mapsto \sqrt{x}$ for the first term and the subadditivity of the same function on the second term (constants merged into $d_1$ and $d_2$, respectively). The existence of the constants $d_1$ and $d_2$ follows from the fact that $\phi_{\max}\del[1]{\Gamma_{J,J}}$ and $\phi_{\min}\del[1]{\Gamma_{J,J}}$ are bounded from above and below, respectively. The last inequality follows by choosing $a_{N,T}$ sufficiently large while the last equality follows from the definitions of $\lambda_{N,T}, \mu_{N,T}$ and  $\xi_{N,T}$ and the fact that $\kappa^2$ is bounded from below.
\end{proof}

\begin{proof}[Proof of Theorem \ref{LassoAsym}]
All notation is as in the statement of Theorem \ref{thm3}. We start with the consistency part. The conclusion follows from Theorem \ref{thm3} if we show that $P\del[0]{\mathcal{A}_{N,T}\cap \mathcal{B}_{N,T}}\to 1$ and that $\xi_{N,T}/\sqrt{NT},\ \xi_{N,T}/\sqrt{T}\to 0$. To establish that $P\del[0]{\mathcal{A}_{N,T}\cap \mathcal{B}_{N,T}}\to 1$ it suffices to show that $A(p^2+Np) e^{-B(t^2N)^{1/3}}\to 0$ Note that, ignoring constants,
\begin{align*}
t^2N=\frac{N}{\del[1]{N^c}^2\del[1]{\frac{N^b}{\ln(N)}\vee \frac{\ln(N)}{N^b}}^6}=N^{1-2c-6b}\ln(N)^6\rightarrow \infty
\end{align*} 
because $6b+2c\leq 9b+2c< 1$. Since $t^2N\to \infty$ and $p$ increases exponentially in $N$ it is enough to show that $p^2e^{-B(t^2N)^{1/3}}\to 0$. But this is the case, since
\begin{align*}
p^2e^{-B(t^2N)^{1/3}}=\exp(2N^b)\exp(-BN^{(1/3-(2/3)c-2b)}\ln(N)^2)\to 0
\end{align*}
because $9b+2c\leq 1$. Next, note that, ignoring constants, $\xi_{N,T}=\log(N)^{3/2}N^{(3/2)b}N^{c/2}+\log(N)^3N^{c/2}$ which implies that
\begin{align*}
\xi_{N,T}/\sqrt{NT}=\log(N)^{3/2}N^{(3/2)b+c/2-1/2-(1/2)a}+\log(N)^3N^{c/2-1/2-(1/2)a}\to 0
\end{align*}
since $3b+c<1+a$. Similarly,
\begin{align*}
\xi_{N,T}/\sqrt{T}=\log(N)^{3/2}N^{(3/2)b+c/2-(1/2)a}+\log(N)^3N^{c/2-(1/2)a}\to 0
\end{align*}
since $3b+c<a$. 

Regarding the second part we have already established that $P\del[0]{\mathcal{A}_{N,T}\cap \mathcal{B}_{N,T}}\to 1$ since $9b+2c< 1$. Hence, $\enVert[0]{\hat{\beta}-\beta^*}\leq \xi_{N,T}/\sqrt{NT}$ with probability tending to one.  But $\hat{\beta}_j=0$ for some $j\in J_1$ implies $\enVert[0]{\hat{\beta}-\beta^*}>\xi_{N,T}/\sqrt{NT}$ from a certain step and onwards. This is a contradiction and so it can't be the case that $\hat{\beta}_j=0$ for any $j\in J_1$. A similar argument applies to $\hat{c}_i$ for $i\in J_2$.
\end{proof}

\begin{lemma}\label{CLp}
Under assumption A1) and A2a) 
\begin{enumerate}
\item $P(\mathcal{C}_{1,N,T})\geq 1-\frac{2}{a_{N,T}^{r/2}}$ for $K_{1,N,T}=|J_1^c|^{2/r}|J_1|^{2/r}(NT)^{1/2}a_{N,T}$
\item $P(\mathcal{C}_{2,N,T})\geq 1-\frac{1}{a_{N,T}^{r}}$ for $K_{2,N,T}=|J_1|^{1/r}|J_2^c|^{1/r}T^{1/2}a_{N,T}$
\end{enumerate}
\end{lemma}

\begin{proof}
First, note that
\begin{align*}
\enVert[3]{\frac{1}{\sqrt{NT}}\sum_{i=1}^{N}\sum_{t=1}^Tx_{i,t,k}x_{i,t,l}}_{L_{r/2}}
&\leq
\sqrt{NT}\max_{1\leq t\leq T}\enVert{x_{1,t,k}x_{1,t,l}}_{L_{r/2}}
\leq
\sqrt{NT}
\end{align*}    
where the last estimate follows from the Cauchy-Schwarz inequality. Hence,\\ $\enVert[1]{\max_{k\in J_1^c}\max_{l\in J_1}\frac{1}{\sqrt{NT}}\sum_{i=1}^{N}\sum_{t=1}^Tx_{i,t,k}x_{i,t,l}}_{L_{r/2}}\leq |J_1^c|^{2/r}|J_1|^{2/r}\sqrt{NT}$. It follows from Markov's inequality that 
\begin{align*}
P\del[2]{\max_{k\in J_1^c}\max_{l\in J_1}\frac{1}{\sqrt{NT}}\sum_{i=1}^{N}\sum_{t=1}^Tx_{i,t,k}x_{i,t,l}
\geq 
K_{1,N,T}}\leq \frac{|J_1^c||J_1|(NT)^{r/4}}{K_{1,N,T}^{r/2}}
=
\frac{1}{a_{N,T}^{r/2}}
\end{align*}
Next, 
\begin{align*}
\enVert[3]{\frac{1}{\sqrt{T}}\sum_{t=1}^Tx_{i,t,k}}_{L_r}
\leq
\sqrt{T}\max_{1\leq t\leq T}\enVert{x_{i,t,k}}_{L_r}
\leq
\sqrt{T}
\end{align*}
This implies, $\enVert[1]{\max_{i\in J_2}\max_{k\in J_1^c}\frac{1}{\sqrt{T}}\sum_{t=1}^Tx_{i,t,k}}_{L_r}\leq |J_1^c|^{1/r}|J_2|^{1/r}\sqrt{T}$ and Markov's inequality yields
\begin{align*}
P\del[2]{\max_{i\in J_2}\max_{k\in J_1^c}\frac{1}{\sqrt{T}}\sum_{t=1}^Tx_{i,t,k}\geq K_{1,N,T}}
\leq
\frac{|J_1^c||J_2|T^{r/2}}{K_{1,N,T}^r}
=
\frac{|J_2|}{|J_1^c||J_1|^2N^{r/2}a_{N,T}^r}
\leq
\frac{1}{a_{N,T}^{r/2}}
\end{align*}
where the last estimates follows from $|J_2|\leq N^{r/2}$ and $a_{N,T}\geq 1$. The conclusion of the first part of the lemma now follows by a union bound. The second part of the lemma is proved in a similar manner.
\end{proof}

\begin{lemma}\label{Cexp}
Under assumption A1) and A2b) 
\begin{enumerate}
\item $P(\mathcal{C}_{1,N,T})\geq 1-\frac{4}{a_{N,T}}$ for $K_{1,N,T}=A\log(1+|J_1^c|)\log(e+|J_1|)\sqrt{NT}\log(a_{N,T})$
\item $P(\mathcal{C}_{2,N,T})\geq 1-\frac{2}{a_{N,T}}$ for $K_{2,N,T}=A\log(1+|J_1|)\log(1+|J_2^c|)\sqrt{T}\log(a_{N,T})$
\end{enumerate}
for a constant $A>0$.
\end{lemma}

\begin{proof}
First, note that 
\begin{align*}
\enVert[3]{\frac{1}{\sqrt{NT}}\sum_{i=1}^{N}\sum_{t=1}^Tx_{i,t,k}x_{i,t,l}}_{\psi_1}
\leq
\sqrt{NT}\max_{1\leq t\leq T}\enVert{x_{1,t,k}x_{1,t,l}}_{\psi_1}
\leq
\sqrt{NT}\frac{1+K}{C}
\end{align*}    
where the last estimate follows from $\enVert{x_{1,t,k}x_{1,t,l}}_{\psi_1}\leq \frac{1+K}{C}:=A$ as argued in the proof of Lemma \ref{Bboundexp}. Hence, $\enVert[1]{\max_{k\in J_1^c}\max_{l\in J_1}\frac{1}{\sqrt{NT}}\sum_{i=1}^{N}\sum_{t=1}^Tx_{i,t,k}x_{i,t,l}}_{\psi_1}\leq A\log(1+|J_1^c|)\log(e+|J_1|)\sqrt{NT}$. By Markov's inequality, the definition of the Orlicz norm, and the fact that $1\wedge \psi(x)^{-1}=1\wedge (e^x-1)^{-1}\leq 2e^{-x}$,
\begin{align*}
&P\del[2]{\max_{k\in J_1^c}\max_{l\in J_1}\frac{1}{\sqrt{NT}}\sum_{i=1}^{N}\sum_{t=1}^Tx_{i,t,k}x_{i,t,l}
\geq K_{1,N,T}}\\
&\leq
1 \wedge \frac{1}{\exp(K_{1,N,T}/A\log(1+|J_1^c|)\log(1+|J_1|)\sqrt{NT})-1}
=
\frac{2}{a_{N,T}}
\end{align*}
Next, since $x_{i,t,k}$ is subgaussian it is also subexponential, and so there exists a constant $A>0$ such that 
\begin{align*}
\enVert[3]{\frac{1}{\sqrt{T}}\sum_{t=1}^Tx_{i,t,k}}_{\psi_1}
\leq
\sqrt{T}\max_{1\leq t\leq T}\enVert{x_{i,t,k}}_{\psi_1}
\leq
\sqrt{T}A
\end{align*}
This implies, $\enVert[1]{\max_{i\in J_2}\max_{k\in J_1^c}\frac{1}{\sqrt{T}}\sum_{t=1}^Tx_{i,t,k}}_{\psi_1}\leq A\log(1+|J_1^c|)\log(1+|J_2|)\sqrt{T}$ and Markov's inequality yields by similar arguments as above\footnote{The constant $A$ may take different values throughout.}
\begin{align*}
&P\del[2]{\max_{i\in J_2}\max_{k\in J_1^c}\frac{1}{\sqrt{T}}\sum_{t=1}^Tx_{i,t,k}\geq K_{1,N,T}}\\
&\leq
1\wedge\frac{1}{\exp(K_{1,N,T}/A\log(1+|J_1^c|)\log(1+|J_2|)\sqrt{T})-1}\\
&\leq
2\exp\del[2]{-\frac{A\log(1+|J_1^c|)\log(e+|J_1|)\sqrt{NT}\log(a_{N,T})
}{A\log(1+|J_1^c|)\log(1+|J_2|)\sqrt{T}}}
\leq
\frac{2}{a_{N,T}}
\end{align*}
where the last estimate follows from $\log(1+|J_2|)\leq N^{1/2}$, $\log(e+|J_1|)\geq 1$ and $a_{N,T}\geq 1$. The conclusion of the first part of the lemma now follows by a union bound. The second part of the lemma is proved in a similar manner.
\end{proof}

Before we prove Theorem \ref{aLassoThm1} note that $\mathcal{\tilde{B}}_{N,T}\subseteq \mathcal{B}_{N,T}$ (see the definition of $\tilde{B}_{N,T}$ in (\ref{tildeB})) as already argued in the proofs of Lemmas \ref{BboundLp} and \ref{Bboundexp}. Furthermore, an argument similar to the one in Lemma \ref{vdGB} reveals that $\mathcal{D}_{N,T}=\cbr[0]{\phi_{\min}(\Psi_{J,J})\geq \frac{1}{2}\phi_{\min}(\Gamma_{J,J})}$ occurs if the maximal entry of $|\Psi_{J,J}-\Gamma_{J,J}|$ is less than or equal to $\frac{\phi_{\min}(\Gamma_{J,J})}{2(s_1+s_2)}$. But this latter event clearly contains $\mathcal{\tilde{B}}_{N,T}$ and so $\mathcal{\tilde{B}}_{N,T}\subseteq \mathcal{D}_{N,T}$.

\begin{proof}[Proof of Theorem \ref{aLassoThm1}]
We shall prove the first part of the theorem since the proof of the second part follows along exactly the same lines (except for replacing $\mathcal{C}_{1,N,T}$ by $\mathcal{C}_{2,N,T}$ in the following arguments). Throughout we work on $\mathcal{A}_{N,T}\cap\mathcal{C}_{1,N,T}\cap \mathcal{D}_{N,T}\cap\cbr[0]{\enVert[0]{\hat{\beta}-\beta^*}\leq \beta_{\min}/2}\cap\cbr[0]{\enVert[0]{\hat{c}-c^*}\leq c_{\min}/2}$ and verify that (\ref{Cond1}) and (\ref{Cond2}) are valid on this set with $w=(w_1',w_2')'$ and $w_{1j}=\lambda_{N,T}/\envert[0]{\hat{\beta}_j},\ j=1,...,p$ as well as $w_{2j}=\mu_{N,T}/\envert[0]{\hat{c}_j}$, $j=1,...,N$ and the convention that $1/0 =\infty$. First note that since $S_{J,J}$ is a diagonal matrix with positive entries on the diagonal (\ref{Cond1}) is equivalent to  
\begin{align*}
\sgn\del[2]{S_{J,J}\gamma^*_J+S_{J,J}\del{Z_J'Z_J}^{-1}S_{J,J}(S_{J,J})^{-1}\sbr{Z_J'\epsilon -r_J}}=\sgn(\gamma^*_J)
\end{align*}
Focussing on an $X_j$ with $j\in J_1$ it hence suffices to show that\footnote{Here, without causing confusion, we assume that $X_j$, $j\in J_1$ is indeed the $j$th variable.} 
\begin{align*}
\envert[1]{\del[1]{S_{J,J}\del{Z_J'Z_J}^{-1}S_{J,J}(S_{J,J})^{-1}\sbr{Z_J'\epsilon -r_J}}_j}
\leq
\sqrt{NT}\beta_{\min}
\end{align*}
The left hand side in the above display may be upper bounded by\\ $\enVert[1]{S_{J,J}\del{Z_J'Z_J}^{-1}S_{J,J}}_{\ell_\infty}\enVert[1]{(S_{J,J})^{-1}\sbr{Z_J'\epsilon -r_J}}_{\ell_\infty}
$. Since 
\begin{align*}
\enVert[0]{S_{J,J}\del{Z_J'Z_J}^{-1}S_{J,J}}_{\ell_\infty}
\leq
\sqrt{|J|}\enVert[0]{S_{J,J}\del{Z_J'Z_J}^{-1}S_{J,J}}
\end{align*}
and on $\mathcal{D}_{N,T}$ one has 
\begin{align}
\enVert[1]{S_{J,J}\del{Z_J'Z_J}^{-1}S_{J,J}}
=
\phi_{\max}\del[1]{S_{J,J}\del{Z_J'Z_J}^{-1}S_{J,J}}
=
\frac{1}{\phi_{\min}(\Psi_{J,J})}
\leq
\frac{2}{\phi_{\min}(\Gamma_{J,J})}\label{inverse}
\end{align}
it follows that 
\begin{align*}
\enVert[0]{S_{J,J}\del{Z_J'Z_J}^{-1}S_{J,J}}_{\ell_\infty}
\leq
\frac{2\sqrt{|J|}}{\phi_{\min}(\Gamma_{J,J})}
\end{align*}
Furthermore, because $\enVert[0]{\hat{\beta}-\beta^*}\leq \beta_{\min}/2$ (by assumption)
\begin{align*}
|\hat{\beta}_j|\geq \beta_j^*-|\hat{\beta}_j-\beta_j^*|\geq \beta_{\min}-\enVert[0]{\hat{\beta}-\beta^*}\geq \beta_{\min}/2
\end{align*}
for all $j\in J_1$. By a similar argument $\hat{c}_j\geq c_{\min}/2$ for all $j\in J_2$. Hence,
\begin{align*}
\enVert[1]{(S_{J,J})^{-1}r_J}_{\ell_\infty}
=
\enVert[2]{\frac{\lambda_{N,T}}{\sqrt{NT}\hat{\beta}_{J_1}}}_{\ell_\infty}\vee \enVert[2]{\frac{\mu_{N,T}}{\sqrt{T}\hat{c}_{J_2}}}_{\ell_\infty}
\leq 
\frac{2\lambda_{N,T}}{\sqrt{NT}\beta_{\min}}\vee \frac{2\mu_{N,T}}{\sqrt{T}c_{\min}}
\end{align*}
Next, on $\mathcal{A}_{N,T}$
\begin{align*}
\enVert[1]{(S_{J,J})^{-1}Z_J'\epsilon}_{\ell_\infty}
\leq
\enVert{\frac{X'_{J_1}\epsilon}{\sqrt{NT}}}_{\ell_\infty}\vee\enVert{\frac{D'_{J_2}\epsilon}{\sqrt{T}}}_{\ell_\infty}
\leq
\frac{\lambda_{N,T}}{2\sqrt{NT}}\vee \frac{\mu_{N,T}}{2\sqrt{T}}
\end{align*}
It follows that
\begin{align}
\enVert[1]{(S_{J,J})^{-1}\sbr{Z_J'\epsilon -r_J}}_{\ell_\infty}
&\leq 
\enVert[1]{(S_{J,J})^{-1}Z_J'\epsilon}_{\ell_\infty}+\enVert[1]{(S_{J,J})^{-1}r_J}_{\ell_\infty}\notag\\
&\leq
\frac{\lambda_{N,T}}{2\sqrt{NT}}\vee \frac{\mu_{N,T}}{2\sqrt{T}}+\frac{2\lambda_{N,T}}{\sqrt{NT}\beta_{\min}}\vee \frac{2\mu_{N,T}}{\sqrt{T}c_{\min}}\label{infty}
\end{align}
Hence, putting the pieces together, (\ref{Cond1}) is satisfied for all $j\in J_1$ if
\begin{align*}
\frac{2\sqrt{|J|}}{\phi_{\min}(\Gamma_{J,J})}\del[3]{\frac{\lambda_{N,T}}{2\sqrt{NT}}\vee \frac{\mu_{N,T}}{2\sqrt{T}}+\frac{2\lambda_{N,T}}{\sqrt{NT}\beta_{\min}}\vee \frac{2\mu_{N,T}}{\sqrt{T}c_{\min}}}
\leq
\sqrt{NT}\beta_{\min}
\end{align*}
Next, (\ref{Cond2}) is equivalent to
\begin{align}
\envert[1]{-Z_{j}'Z_J(S_{J,J})^{-1}S_{J,J}\del{Z_J'Z_J}^{-1}S_{J,J}(S_{J,J})^{-1}\sbr{Z_J'\epsilon -r_J}+Z_{j}'\epsilon}< w_j\label{LHS2}
\end{align}
for all $j\in J^c$. The left hand side in the above display is bounded from above by
\begin{align*}
\enVert[1]{Z_{j}'Z_J(S_{J,J})^{-1}S_{J,J}\del{Z_J'Z_J}^{-1}S_{J,J}}_{\ell_1}\enVert[1]{(S_{J,J})^{-1}\sbr{Z_J'\epsilon -r_J}}_{\ell_\infty}+\envert[1]{Z_j'\epsilon}
\end{align*}
Assume again that $Z_j$ is an $X_j$. Then, on $\mathcal{C}_{1,N,T}$ and by (\ref{inverse})
\begin{align*}
&\enVert[1]{Z_{j}'Z_J(S_{J,J})^{-1}S_{J,J}\del{Z_J'Z_J}^{-1}S_{J,J}}_{\ell_1}
\leq 
\sqrt{|J|}\enVert[1]{Z_{j}'Z_J(S_{J,J})^{-1}S_{J,J}\del{Z_J'Z_J}^{-1}S_{J,J}}\\
&\leq
|J|\enVert[1]{Z_{j}'Z_J(S_{J,J})^{-1}}_{\ell_\infty}\enVert[1]{S_{J,J}\del{Z_J'Z_J}^{-1}S_{J,J}}\\
&\leq
\frac{2|J|K_{1,N,T}}{\phi_{\min}(\Gamma_{J,J})}
\end{align*}
where the second estimate follows by considering $S_{J,J}\del{Z_J'Z_J}^{-1}S_{J,J}$ as a bounded linear operator from $\ell_2(\mathbb{R}^{|J|})\to \ell_2(\mathbb{R}^{|J|})$ with induced operator norm given by $\phi_{\max}\del[1]{S_{J,J}\del{Z_J'Z_J}^{-1}S_{J,J}}$. Putting the pieces together, and using that we are on $\mathcal{A}_{N,T}$ and by (\ref{infty}) the left hand side in (\ref{LHS2}) may be upper bounded by
\begin{align*}
\frac{2|J|K_{1,N,T}}{\phi_{\min}(\Gamma_{J,J})}\del[3]{\frac{\lambda_{N,T}}{2\sqrt{NT}}\vee \frac{\mu_{N,T}}{2\sqrt{T}}+\frac{2\lambda_{N,T}}{\sqrt{NT}\beta_{\min}}\vee \frac{2\mu_{N,T}}{\sqrt{T}c_{\min}}}+\frac{\lambda_{N,T}}{2}
\end{align*}
Finally, the right hand side in (\ref{LHS2}) may be bounded from below by $\lambda_{N,T}/\enVert[0]{\hat{\beta}-\beta^*}$ and the result follows.
\end{proof}

\begin{proof}[Proof of Corollary \ref{ALcor}]
We know from Theorem \ref{aLassoThm1} that $\sgn({\tilde{\beta}})=\sgn({\beta^*})$ on $\mathcal{A}_{N,T}\cap\mathcal{B}_{N,T}\cap\mathcal{C}_{1,N,T}\cap\mathcal{D}_{N,T}\cap\cbr[0]{\enVert[0]{\hat{\beta}-\beta^*}\leq \beta_{\min}/2}\cap\cbr[0]{\enVert[0]{\hat{c}-c^*}\leq c_{\min}/2}                                                                                   $ if (\ref{albeta1})-(\ref{albeta2}) is satisfied\footnote{Actually, we know from Theorem \ref{aLassoThm1} that $\sgn({\tilde{\beta}})=\sgn({\beta^*})$ on the larger set $\mathcal{A}_{N,T}\cap\mathcal{C}_{1,N,T}\cap\mathcal{D}_{N,T}\cap\cbr[0]{\enVert[0]{\hat{\beta}-\beta^*}\leq \beta_{\min}/2}\cap\cbr[0]{\enVert[0]{\hat{c}-c^*}\leq c_{\min}/2}                                                                                   $. As will be seen, this distinction will turn out not to make any difference for our lower bounds on the probability of the events.}. Furthermore, if $\beta_{\min}\geq\frac{2\xi_{N,T}}{\sqrt{NT}}$ one has $\mathcal{A}_{N,T}\cap\mathcal{\tilde{B}}_{N,T}\cap\mathcal{C}_{1,N,T}\subseteq\mathcal{A}_{N,T}\cap\mathcal{B}_{N,T}\mathcal{C}_{1,N,T}\cap\mathcal{D}_{N,T}\cap\cbr[0]{\enVert[0]{\hat{\beta}-\beta^*}\leq \beta_{\min}/2}\cap \cbr[0]{\enVert[0]{\hat{c}-c^*}\leq c_{\min}/2} $ \footnote{The inclusion follows from the fact that $\mathcal{\tilde{B}}_{N,T}\subseteq \mathcal{B}_{N,T}\cap \mathcal{D}_{N,T}$ as argued prior to the proof of Theorem \ref{aLassoThm1}. Also, the inclusion has used that on $\mathcal{A}_{N,T}\cap \mathcal{B}_{N,T}$ one has $\enVert[0]{\hat{\beta}-\beta^*}\leq \frac{\xi_{N,T}}{\sqrt{NT}}$ and $\enVert[0]{\hat{c}-c^*}\leq \frac{\xi_{N,T}}{\sqrt{T}}$ such that $\beta_{\min}\geq\frac{2\xi_{N,T}}{\sqrt{NT}}$ and $c_{\min}\geq\frac{2\xi_{N,T}}{\sqrt{T}}$ imply that $\cbr[0]{\enVert[0]{\hat{\beta}-\beta^*}\leq \beta_{\min}/2}$ and $\cbr[0]{\enVert[0]{\hat{c}-c^*}\leq c_{\min}/2}$, respectively.}. The lower bound on the probability of $\cbr[0]{\sgn({\tilde{\beta}})=\sgn({\beta^*})}$ now follows by Lemmas \ref{AboundLp}, \ref{BboundLp} and \ref{CLp} in case of part one of the corollary. In case of part 2 of the corollary Lemmas \ref{Aboundexp}, \ref{Bboundexp} and \ref{Cexp} are used. A similar argument gives the lower bound on the probability with which $\sgn(\tilde{c})=\sgn(c^*)$ by verifying (\ref{alc1})-(\ref{alc2}). 
\end{proof}

\begin{proof}[Proof of Theorem \ref{aLassoAsym}]
We proceed by verifying the conditions for the sign consistency of $\tilde{\beta}$  and $\tilde{c}$ in part 2 of Corollary \ref{ALcor} and show that the lower bound on the probability with which $\sgn(\tilde{\beta})=\sgn(\beta^*)$ and $\sgn(\tilde{c})=\sgn(c^*)$ tends to one. We focus on $P\del[1]{\sgn(\tilde{\beta})=\sgn(\beta^*)}\to 1$ since the second part of the theorem follows by identical arguments. 

First, we verify that (\ref{albeta1}) is satisfied asymptotically. To do so it suffices to show that $\frac{\sqrt{|J|}\lambda_{N,T}}{NT}\to 0$ and $\frac{\sqrt{|J|}\mu_{N,T}}{\sqrt{N}T}\to 0$. Now, ignoring constants, and using the definition of $\lambda_{N,T}$
\begin{align*}
\frac{\sqrt{|J|}\lambda_{N,T}}{NT}
=
\frac{\sqrt{|J|}\log(p)^{3/2}\log({a_{N,T}})^{3/2}}{\sqrt{NT}}
=
N^{c/2+\frac{3}{2}b-a/2-1/2}\log{N}^{3/2}\to 0
\end{align*}
since $3b+c<1+a$. Similarly, using the definition of $\mu_{N,T}$
\begin{align*}
\frac{\sqrt{|J|}\mu_{N,T}}{\sqrt{N}T}
=
\frac{\sqrt{|J|}\log(N)^3}{\sqrt{NT}}
=
N^{c/2-a/2-1/2}\log(N)^3\to 0
\end{align*}
since $c<a+1$. Next, we verify that (\ref{albeta2}) is valid asymptotically. To do so it suffices to show that $\frac{|J|K_{1,N,T}}{\sqrt{NT}}\enVert[0]{\hat{\beta}-\beta^*}\to 0$, $\frac{|J|K_{1,N,T}\mu_{N,T}/\lambda_{N,T}}{\sqrt{T}}\enVert[0]{\hat{\beta}-\beta^*}\to 0$	and $\enVert[0]{\hat{\beta}-\beta^*}\to 0$. For this purpose, note that $K_{1,N,T}\leq A\log(1+|p|)\log(e+|J_1|)\sqrt{NT}\log(a_{N,T})$ which is of order $\log(|p|)\log(|J_1|)\sqrt{NT}\log(a_{N,T})=N^bc\log(N)^2\sqrt{NT}=N^{b+1/2+a/2}\log(N)^2$. Furthermore, $\enVert[0]{\hat{\beta}-\beta^*}\leq \xi_{N,T}/\sqrt{NT}$ on $\mathcal{A}_{N,T}\cap \mathcal{B}_{N,T}$ which we are working on in Corollary \ref{ALcor} \footnote{See the first line of the proof of Corollary \ref{ALcor}.} (where $\xi_{N,T}$ is as defined in Theorem \ref{thm3}). Hence, ignoring constants,
\begin{align}
\enVert[0]{\hat{\beta}-\beta^*}
\leq
\xi_{N,T}/\sqrt{NT}
\leq
\log(N)^3N^{\frac{3}{2}b+c/2-1/2-a/2}\to 0 \label{xiaux}
\end{align} 
since $3b+c<1+a$. Also,
\begin{align*}
\frac{|J|K_{1,N,T}}{\sqrt{NT}}\enVert[0]{\hat{\beta}-\beta^*}
&\leq
N^{c}N^{b+1/2+a/2}\log(N)^2\log(N)^3N^{\frac{3}{2}b+c/2-1/2-a/2}N^{-1/2-a/2}\\
&=
N^{\frac{5}{2}b+\frac{3}{2}c-1/2-a/2}\log(N)^5\to 0
\end{align*}
since $5b+3c<1+a$. Similarly, since $\mu_{N,T}/\lambda_{N,T}=\frac{\log(N)^{3/2}}{\sqrt{N}\log(p)^{3/2}}=\frac{\log(N)^{3/2}}{\sqrt{N}N^{\frac{3}{2}b}}$
\begin{align*}
\frac{|J|K_{1,N,T}\mu_{N,T}/\lambda_{N,T}}{\sqrt{T}}\enVert[0]{\hat{\beta}-\beta^*}
&=
N^cN^{b+1/2+a/2}\log(N)^2\frac{\log(N)^{3/2}}{\sqrt{N}N^{\frac{3}{2}b}}\log(N)^3N^{\frac{3}{2}b+c/2-1/2-a/2}N^{-a/2}\\
&=
N^{b+\frac{3}{2}c-1/2-a/2}\log(N)^{13/2}\to 0
\end{align*}
since $2b+3c<1+a$. Furthermore, $\beta_{\min}\geq 2\frac{\xi_{N,T}}{\sqrt{NT}}$ since $\frac{\xi_{N,T}}{\sqrt{NT}}\to 0$ when $3b+c<1+a$ as seen from (\ref{xiaux}) while $\beta_{\min}$ is bounded away from 0. Finally, we note that $9b+2c\leq 1$ suffices to ensure that the lower bound on the probability in part 2 of Corollary \ref{ALcor} tends to one as was already argued in the proof of Theorem \ref{LassoAsym}.
\end{proof}

\bibliographystyle{chicagoa}	
\bibliography{references}		

\begin{thebibliography}{}

\bibitem[\protect\citeauthoryear{Arellano}{Arellano}{2003}]{arellano03}
Arellano, M. (2003).
\newblock {\em Panel Data Econometrics}, Volume~1.
\newblock Oxford University Press, Oxford.


\bibitem[\protect\citeauthoryear{Barro}{Barro}{1991}]{barro91}
Barro, R.~J. (1991).
\newblock Economic growth in a cross section of countries.
\newblock {\em The Quarterly Journal of Economics\/}~{\em 106\/}(2), 407--443.


\bibitem[\protect\citeauthoryear{Belloni, Chen, Chernozhukov, and
  Hansen}{Belloni et~al.}{2012}]{belloni2012s}
Belloni, A., D.~Chen, V.~Chernozhukov, and C.~Hansen (2012).
\newblock Sparse models and methods for optimal instruments with an application
  to eminent domain.
\newblock {\em Econometrica\/}~{\em 80\/}(6), 2369--2429.


\bibitem[\protect\citeauthoryear{Belloni and Chernozhukov}{Belloni and
  Chernozhukov}{2011}]{bellonic11}
Belloni, A. and V.~Chernozhukov (2011).
\newblock High dimensional sparse econometric models: An introduction.
\newblock {\em Inverse Problems and High-Dimensional Estimation\/}, 121--156.


\bibitem[\protect\citeauthoryear{Belloni, Chernozhukov, and Hansen}{Belloni
  et~al.}{2013}]{BelloniCH13}
Belloni, A., V.~Chernozhukov, and C.~Hansen (2013).
\newblock Inference on treatment effects after selection amongst
  high-dimensional controls.
\newblock {\em The Review of Economic Studies\/}.


\bibitem[\protect\citeauthoryear{Bickel, Ritov, and Tsybakov}{Bickel
  et~al.}{2009}]{bickelrt09}
Bickel, P., Y.~Ritov, and A.~Tsybakov (2009).
\newblock Simultaneous analysis of lasso and dantzig selector.
\newblock {\em The Annals of Statistics\/}~{\em 37\/}(4), 1705--1732.


\bibitem[\protect\citeauthoryear{B{\"u}hlmann and van~de Geer}{B{\"u}hlmann and
  van~de Geer}{2011}]{buhlmannvdg11}
B{\"u}hlmann, P. and S.~van~de Geer (2011).
\newblock {\em Statistics for High-Dimensional Data: Methods, Theory and
  Applications}.
\newblock Springer-Verlag, New York.


\bibitem[\protect\citeauthoryear{Caner}{Caner}{2009}]{caner2009}
Caner, M. (2009).
\newblock Lasso-type gmm estimator.
\newblock {\em Econometric Theory\/}~{\em 25\/}(1).


\bibitem[\protect\citeauthoryear{Caner and Zhang}{Caner and
  Zhang}{2014}]{canerz2014}
Caner, M. and H.~H. Zhang (2014).
\newblock Adaptive elastic net for generalized methods of moments.
\newblock {\em Journal of Business \& Economic Statistics (forthcoming)\/}.


\bibitem[\protect\citeauthoryear{Fan, Lv, and Qi}{Fan et~al.}{2011}]{fan11}
Fan, J., J.~Lv, and L.~Qi (2011).
\newblock Sparse high dimensional models in economics.
\newblock {\em Annual review of economics\/}~{\em 3}, 291.


\bibitem[\protect\citeauthoryear{Hall and Heyde}{Hall and
  Heyde}{1980}]{hallh80}
Hall, P. and C.~Heyde (1980).
\newblock {\em Martingale limit theory and its application}, Volume 142.
\newblock Academic press New York.


\bibitem[\protect\citeauthoryear{Hitczenko}{Hitczenko}{1990}]{hitczenko90}
Hitczenko, P. (1990).
\newblock Best constants in martingale version of rosenthal's inequality.
\newblock {\em The Annals of Probability\/}, 1656--1668.


\bibitem[\protect\citeauthoryear{Hsiao}{Hsiao}{2003}]{hsiao03}
Hsiao, C. (2003).
\newblock {\em Analysis of panel data}, Volume~34.
\newblock Cambridge University Press.


\bibitem[\protect\citeauthoryear{Kock}{Kock}{2013}]{kock13}
Kock, A.~B. (2013).
\newblock Oracle efficient variable selection in random and fixed effects panel
  data models.
\newblock {\em Econometric Theory\/}~{\em 29\/}(01), 115--152.


\bibitem[\protect\citeauthoryear{Kock and Callot}{Kock and
  Callot}{2013}]{kock2013oracle}
Kock, A.~B. and L.~A. Callot (2013).
\newblock Oracle inequalities for high dimensional vector autoregressions.
\newblock {\em arXiv preprint arXiv:1311.0811\/}.


\bibitem[\protect\citeauthoryear{Lesigne and Voln{\`y}}{Lesigne and
  Voln{\`y}}{2001}]{lesignev01}
Lesigne, E. and D.~Voln{\`y} (2001).
\newblock Large deviations for martingales.
\newblock {\em Stochastic processes and their applications\/}~{\em 96\/}(1),
  143--159.


\bibitem[\protect\citeauthoryear{Meinshausen and B{\"u}hlmann}{Meinshausen and
  B{\"u}hlmann}{2006}]{meinshausenb06}
Meinshausen, N. and P.~B{\"u}hlmann (2006).
\newblock High-dimensional graphs and variable selection with the lasso.
\newblock {\em The Annals of Statistics\/}~{\em 34}, 1436--1462.


\bibitem[\protect\citeauthoryear{P{\"o}tscher and Leeb}{P{\"o}tscher and
  Leeb}{2009}]{potscherl09}
P{\"o}tscher, B.~M. and H.~Leeb (2009).
\newblock On the distribution of penalized maximum likelihood estimators: The
  lasso, scad, and thresholding.
\newblock {\em Journal of Multivariate Analysis\/}~{\em 100\/}(9), 2065--2082.


\bibitem[\protect\citeauthoryear{Raskutti, Wainwright, and Yu}{Raskutti
  et~al.}{2010}]{raskuttiwy10}
Raskutti, G., M.~J. Wainwright, and B.~Yu (2010).
\newblock Restricted eigenvalue properties for correlated gaussian designs.
\newblock {\em The Journal of Machine Learning Research\/}~{\em 11},
  2241--2259.


\bibitem[\protect\citeauthoryear{Rudelson and Zhou}{Rudelson and
  Zhou}{2011}]{rudelsonz2013}
Rudelson, M. and S.~Zhou (2011).
\newblock Reconstruction from anisotropic random measurements.


\bibitem[\protect\citeauthoryear{Stoianov}{Stoianov}{1997}]{stoianov87}
Stoianov, {\u{I}}. (1997).
\newblock {\em Counterexamples in probability}.
\newblock John Wiley \& Sons (Chichester).


\bibitem[\protect\citeauthoryear{Tibshirani}{Tibshirani}{1996}]{tibshirani96}
Tibshirani, R. (1996).
\newblock Regression shrinkage and selection via the lasso.
\newblock {\em Journal of the Royal Statistical Society. Series B
  (Methodological)\/}, 267--288.


\bibitem[\protect\citeauthoryear{van~der Vaart and Wellner}{van~der Vaart and
  Wellner}{1996}]{vdVW96}
van~der Vaart, A.~W. and J.~A. Wellner (1996).
\newblock {\em Weak convergence and empirical processes}.
\newblock Springer Verlag.


\bibitem[\protect\citeauthoryear{Vershynin}{Vershynin}{2011}]{vershynin11}
Vershynin, R. (2011).
\newblock Introduction to the non-asymptotic analysis of random matrices.
\newblock {\em arXiv preprint\/}.


\bibitem[\protect\citeauthoryear{Wainwright}{Wainwright}{2009}]{wainwright09}
Wainwright, M. (2009).
\newblock Sharp thresholds for high-dimensional and noisy sparsity recovery
  using l1-constrained quadratic programming (lasso).
\newblock {\em Information Theory, IEEE Transactions on\/}~{\em 55\/}(5),
  2183--2202.


\bibitem[\protect\citeauthoryear{Wooldridge}{Wooldridge}{2002}]{wooldridge02}
Wooldridge, J. (2002).
\newblock {\em Econometric analysis of cross section and panel data}.
\newblock The MIT press.


\bibitem[\protect\citeauthoryear{Zhao and Yu}{Zhao and Yu}{2006}]{zhaoy06}
Zhao, P. and B.~Yu (2006).
\newblock On model selection consistency of lasso.
\newblock {\em The Journal of Machine Learning Research\/}~{\em 7}, 2541--2563.


\bibitem[\protect\citeauthoryear{Zhou, van~de Geer, and B{\"u}hlmann}{Zhou
  et~al.}{2009}]{zhouvdGB09}
Zhou, S., S.~van~de Geer, and P.~B{\"u}hlmann (2009).
\newblock Adaptive lasso for high dimensional regression and gaussian graphical
  modeling.
\newblock {\em arXiv preprint arXiv:0903.2515\/}.


\bibitem[\protect\citeauthoryear{Zou}{Zou}{2006}]{zou06}
Zou, H. (2006).
\newblock The adaptive lasso and its oracle properties.
\newblock {\em Journal of the American statistical association\/}~{\em
  101\/}(476), 1418--1429.


\end{thebibliography}


\end{document}